\documentclass[leqno,11pt,a4paper,oneside]{amsart}

\usepackage{geometry}

\usepackage[english]{babel}
\usepackage[T1]{fontenc}

\usepackage[all,pdf]{xy}
\usepackage[babel]{csquotes}
\usepackage[font=small]{caption}
\usepackage[font={footnotesize}]{caption}
\usepackage[export]{adjustbox}
\usepackage[ruled]{algorithm2e}
\usepackage{amscd}
\usepackage{amsmath}
\usepackage{amssymb}
\usepackage{amsthm}
\usepackage{bbm}
\usepackage{booktabs}
\usepackage{colortbl}
\usepackage{enumerate}
\usepackage{fancyhdr}
\usepackage{float}
\usepackage{graphicx}
\usepackage{hyperref}
\usepackage{indentfirst}
\usepackage{lmodern}
\usepackage{longtable}
\usepackage{mathabx}
\usepackage{mathtools}
\usepackage{mleftright}
\usepackage{pgfplots}
\usepackage{subcaption}
\usepackage{tikz,tikz-3dplot}
\usepackage{xurl}

\usetikzlibrary{arrows}
\usetikzlibrary{matrix}
\usetikzlibrary{shapes.misc}

\usepackage{etoolbox}
\apptocmd{\sloppy}{\hbadness 10000\relax}{}{}

\newcommand{\C}{\mathbb{C}}
\newcommand{\N}{\mathbb{N}}

\newcommand{\R}{\mathbb{R}}
\newcommand{\Z}{\mathbb{Z}}

\newcommand{\Pol}{\mathcal{P}}
\newcommand{\Sim}{\mathcal{S}}

\newcommand{\A}{\mathcal{A}}
\newcommand{\Q}{\mathcal{Q}}
\newcommand{\V}{\mathcal{V}}
\newcommand{\W}{\mathcal{W}}
\newcommand{\X}{\mathcal{X}}

\newcommand{\D}{\Delta}

\newcommand{\bzero}{\textnormal{\textbf{0}}}

\newcommand{\be}{\textnormal{\textbf{e}}}
\newcommand{\bp}{\textnormal{\textbf{p}}}
\newcommand{\bv}{\textnormal{\textbf{v}}}

\newcommand{\conv}{\mathrm{conv}}
\newcommand{\ehr}{\mathrm{ehr}}
\newcommand{\h}{h^*}
\newcommand{\Pyr}{\mathrm{Pyr}}

\newcommand{\vol}{\mathrm{vol}}
\newcommand{\Vol}{\mathrm{Vol}}

\newcommand{\orow}{\rowcolor[gray]{0.95}}
\newcommand{\erow}{}

\newtheorem{thm}{Theorem}[section]

\newtheorem{conj}[thm]{Conjecture}
\newtheorem{cor}[thm]{Corollary}
\newtheorem{lem}[thm]{Lemma}
\newtheorem{prop}[thm]{Proposition}
\newtheorem{question}[thm]{Question}

\theoremstyle{definition}
\newtheorem{defn}[thm]{Definition}
\newtheorem{ex}[thm]{Example}

\date{}

\begin{document}

\author[G.\,Balletti]{Gabriele Balletti}
\address{
    Department of Mathematics\\
    Stockholm University\\
    SE-$106$\ $91$\ Stockholm\\
    Sweden
}
\email{balletti@math.su.se}


\title{Enumeration of lattice polytopes by their volume}

\maketitle

\begin{abstract}
	A well known result by Lagarias and Ziegler states that there are
	finitely many equivalence classes of $d$-dimensional lattice polytopes
	having volume at most $K$, for fixed constants $d$ and $K$.
	We describe an algorithm for the complete enumeration of such equivalence
	classes for arbitrary constants $d$ and $K$.
	The algorithm, which gives another proof of the finiteness result, is
	implemented for small values of $K$, up to dimension six.
	The resulting database contains and extends several existing ones, and has
    been used to correct mistakes in other classifications.
    When specialized to three-dimensional smooth polytopes, it extends previous
    classifications by Bogart et al., Lorenz and Lundman.
    Moreover, we give a structure theorem for smooth polytopes with few lattice
    points that proves that they have a quadratic triangulation and that we use,
    together with the classification, to describe smooth polytopes having small
    volume in arbitrary dimension.
    In dimension three we enumerate all the simplices having up to 11 interior
    lattice points and we use them to conjecture a set of sharp inequalities
    for the coefficients of the Ehrhart $\h$-polynomials, unifying several
    existing conjectures.
    Finally, we extract and discuss minimal interesting examples from the
    classification, and we study the frequency of properties such as being
    spanning, very ample, IDP, and having a unimodular cover or triangulation.
    In particular, we find the smallest polytopes which are very ample but not
    IDP, and with a unimodular cover but without a unimodular triangulation. 
\end{abstract}

\section{Introduction}
\label{sec:intro}

Finiteness results are not uncommon in the study of lattice polytopes. Most of 
these are proven by fixing the dimension, showing an upper bound for the volume
and then using the following result by Lagarias and Ziegler.

\begin{thm}[{\cite[Theorem~2]{LZ91}}]\label{thm:LZ}
    Up to unimodular equivalence, there are finitely many $d$-dimensional 
    lattice polytopes having volume lower than a constant $K$.
\end{thm}

Note that working up to unimodular equivalence, i.e. up to affine lattice
preserving maps in $GL_d(\Z) \times \Z^d$ is an obvious requirement that 
we will often avoid to mention.

Once it is known that a family of lattice polytopes is finite, it is tempting
to give a complete description of it.
Most of the times this seems not to be possible in full generality, and it is 
instead done explicitly only fixing ``small enough'' parameters, first and 
foremost the dimension. 
A well-known example of finiteness result is the finiteness of $d$-dimensional 
lattice polytopes having a fixed positive number of interior lattice points, 
which follows from a volume bound proven by Hensley \cite{Hen83}. 
This result paved the way to explicit classifications of families of lattice 
polytopes having a fixed number of interior lattice points. 
The best example is probably the massive classification of reflexive polytopes 
(which have one interior lattice point) up to dimension four performed by 
Kreuzer and Skarke to study mirror symmetric Calabi-Yau manifolds 
\cite{KS98,KS00}. 
Another example of polytopes having exactly one interior point are the smooth 
Fano polytopes, fully enumerated up to dimension nine 
\cite{Bat81,WW82,Bat99,Sat00,KN09,Obr07,LP08}. 
Without additional restrictions, lattice polytopes having one and two interior 
lattice points are classified in dimension three \cite{Kas10,BK16}.

In the last years, many other examples of these kind of results were proven.
In \cite{AWW11,AKW17} is proven that in each dimension there are finitely many 
hollow lattice polytopes which are maximal up to inclusion, and they are 
classified in dimension three. 
In \cite{BHH15} the finiteness of smooth polytopes having fixed number of 
lattice points is shown in each dimension. 
Such polytopes are enumerated in dimension three, up to 16 lattice points 
\cite{Lor10,Lun13}. 
A finiteness result in dimension three for polytopes of width larger than one 
and fixed number of interior points is proven in \cite{BS16a}, and an explicit 
enumeration has been performed up to 11 lattice points \cite{BS16b,BS17}.

All the aforementioned results are proven via bounding the volume of the 
considered family of polytopes and applying Theorem~\ref{thm:LZ}. 
In this paper we take a natural step, and use Lagarias and Ziegler's Theorem 
to perform a systematic enumeration of all lattice polytopes of fixed 
dimension and volume that are within computational reach. This is done by 
giving an alternative proof of Theorem~\ref{thm:LZ} that has the advantage 
of being efficiently implementable. The key is to build the polytopes ``from 
below'', starting from a simplex and progressively adding vertices, instead of 
``from above'' as the original proof of Theorem~\ref{thm:LZ} suggest, 
carving out all the possible lattice polytopes from a big cube.

A result with a similar taste, but only in dimension two, has been achieved by 
Castryck in \cite{Cas12}. With a ``moving out'' technique, he gives a different
proof of a finiteness result and classifies all lattice polygons having up to 
30 interior lattice points.

The paper has the following structure. In Section~\ref{sec:point_configurations}
we give an introduction to point configurations, building the necessary tools for the
rest of the paper. In Section~\ref{sec:LZproof} we give an independent proof of
Theorem~\ref{thm:LZ}, which leads to an algorithm, whose implementation is 
discussed in Section~\ref{sec:implementation}. In Section~\ref{sec:results} we discuss
the results of the classification and we compare it with existing ones.
In Section~\ref{sec:smooth} we specialize the classification to smooth
polytopes. In this settings our classification extends other ones performed
in dimension three, and creates a database of small smooth polytopes. We give
a structure theorem for $d$-dimensional smooth polytopes having at most $3d-4$
lattice points (Theorem~\ref{thm:smooth_few_points}), which we immediately 
apply to fully categorize smooth polytopes having normalize volume at most 10
(Proposition~\ref{prop:small_smooth}).
In Section~\ref{sec:3-simplices} we use the classification together with 
existing volume bounds for simplices with a fixed number of interior lattice
points to classify all the three-dimensional lattice simplices having up
to 11 interior lattice points (Corollary~\ref{cor:simplices}).
In this section, we use the classification of the precedent section to
give conjectural Ehrhart inequalities for three-dimensional lattice
polytopes (Conjecture~\ref{conj:main}).
In the final Section~\ref{sec:examples}, we look for interesting examples
contained in the database, and we observe the commonness of polytopes which 
are spanning, very ample, IDP, have a unimodular cover/triangulation
(see Appendix~\ref{app:tables}).

This project originated from the following question that Christian Haase
posed for the participants of the ``Workshop on Convex Polytopes for Graduate
Students'' in Osaka, during January 2017.

\begin{quotation}
``How many 6-dimensional lattice polytopes of volume 5 have the integer-decomposition property?''
\end{quotation}
From the tables in Appendix~\ref{app:tables} we can read that the answer is 27.

I would like to thank my PhD advisor Benjamin Nill for all the inspiring and
patient discussions.
The idea of Algorithm~\ref{alg:pol} came out during one of those.
I am also grateful to Al Kasprzyk for teaching me how to use Magma and Paco Santos
for helpful remarks.
The author is partially supported by the Vetenskapsr{\aa}det grant~NT:2014-3991

\section{Invitation to point configurations and volume vectors}
\label{sec:point_configurations}

In this section we sketch some basic concepts and results regarding point 
configurations, triangulations and volume vectors of lattice polytopes. 
This material will be used in Section~\ref{sec:LZproof} for proving Theorem~\ref{thm:LZ}. 
We use \cite[Chapter~4]{DRS10} as a reference, but we also refer to 
\cite[Chapter~6]{LZ91} for details.

A \emph{point configuration} is a finite set of points $\A$ in an affine space 
$\R^d$. 
We say that a point configuration $\A$ is \emph{independent} if none of its 
points is an affine combination of the rest, otherwise we say that $\A$ is 
\emph{dependent}. 
A point configuration has \emph{corank one} if it has a unique (up to scalar 
multiplication) dependence relation 
$\sum_{\bp \in \A} \lambda_\bp \bp = \bzero$. 
Such dependence relation defines a partition of $\A$ in the sets
{\small
\[
    J_+ \coloneq \left\{\bp \in \A : \lambda_\bp > 0      \right\}, \quad
    J_0 \coloneq \left\{\bp \in \A : \lambda_\bp = \bzero \right\}, \quad
    J_- \coloneq \left\{\bp \in \A : \lambda_\bp < \bzero \right\}.
\]
}
Such partition is unique, up to switching $J_+$ with $J_-$. 
Given any point configuration $\A$ one can consider the polytope $P$ defined 
as the convex hull of the points of $\A$, i.e $P_\A \coloneq \conv(\A)$. 
If $\A$ has corank one then $P_\A$ has exactly two different triangulations in 
simplices having vertices on $A$.

\begin{lem}[{\cite[Lemma 2.4.2]{DRS10}}]
    If a point configuration $\A$ has corank one, then the following are the 
    only two triangulations of $P_\A$ in simplices having vertices in $\A$:
    \[
        \mathcal{T}_+\coloneq\left\{C \subset \A : J_+ \not\subseteq C\right\},
        \qquad \text{and} \qquad
        \mathcal{T}_-\coloneq\left\{C \subset \A : J_- \not\subseteq C\right\}.
    \]
\end{lem}

Note that, supposing $P_\A$ full-dimensional, the full-dimensional simplices in
$\mathcal{T}_+$ and $\mathcal{T}_-$ are $|J_+|$ and $|J_-|$ respectively, i.e. 
the signature of $\A$ is the pair of number of simplices of the two ways to 
triangulate $P_\A$ discussed above.
If the unique dependence relation of a corank one point configuration $\A$ is 
entirely supported on $\A$, i.e. if $J_0$ is empty, then we say that $\A$ is a 
\emph{circuit}. 
In particular every proper subset of a circuit is an independent point 
configuration. 
If $\A$ is a circuit the pair $(J_+,J_-)$ is classically called the 
\emph{Radon partition} or the \emph{oriented circuit} of $\A$.

Being interest in lattice polytopes, we move our focus to point configurations 
contained in the lattice $\Z^d$. 
Given a full-dimensional lattice polytope $P$ in $\R^d$, we denote by 
$\Vol(P)$ its normalized volume $\Vol(P) \coloneq d! \vol(P)$, where $\vol(P)$ 
is the standard euclidean volume. 
In other words we set the normalized volume of any lattice simplex $S$ with 
vertices $\bv_1,\ldots,\bv_{d+1}$ to be
\[
    \Vol(S) \coloneq \left| \det
    \begin{pmatrix}
         1	& \dots	& 1 \\
    \bv_1	& \dots	& \bv_{d+1} \\
    \end{pmatrix} \right|,
\]
Then the notion of volume can be extended to arbitrary polytopes via 
triangulations. 
The notion of volume can be extended in a finer way to polytopes, via volume 
vectors. 
For this we agree with the notation used in \cite{BS16a}. 

\begin{defn}
    Let $\A = \{ \bp_1 , \ldots , \bp_n \}$ in $\Z^d$, with $n \geq d+1$. 
    Then the \emph{volume vector} of $\A$ is defined as
    \[
        (w_{i_1,\ldots,i_{d+1}})_{1 \leq i_1 < \cdots < i_{d+1} \leq n }
        \in \Z^{\binom{n}{d+1}}
    \]
    where
    \[
        (w_{i_1,\ldots,i_{d+1}}) \coloneq \det
        \begin{pmatrix}
                1	& \dots	& 1 \\
        \bp_{i_1}	& \dots	& \bp_{i_{d+1}} \\
        \end{pmatrix}.
    \]
\end{defn}

Note that we are assuming that $\A$ has an intrinsic order on the elements.
The volume vector is a powerful invariant,
which almost encodes all the data of a point configuration.

\begin{prop}[{\cite[Proposition~2.2]{BS16a}}]
    Let $\A$ and $\A'$ be the point configurations $\A=\{\bp_1,\ldots,\bp_n\}$ 
    and $\A'=\{\bp_1',\ldots,\bp_n'\}$ in $\Z^d$, and suppose that (with 
    respect to a given ordering) they have the same volume vector 
    $(w_{i_1,\ldots,i_{d+1}})_{1 \leq i_1 < \cdots < i_{d+1} \leq n } \in 
    \Z^{\binom{n}{d+1}}$. 
    Then
    \begin{enumerate}[(1)]
        \item There is a unique unimodular affine map $t : \R^d \to R^d$ with 
              $t(\A) = \A'$ (respecting the order of points).
        \item If $\gcd_{1 \leq i_1 < \cdots < i_{d+1} \leq n } 
              (w_{i_1,\ldots,i_{d+1}}) = 1$, then $t$ is a $\Z$-equivalence 
              between $\A$ and $\A'$. 
              In particular $P_\A$ and $P_{\A'}$ are unimodular equivalent 
              lattice polytopes.
    \end{enumerate}
\end{prop}

We now restrict our interest to point configurations in $\Z^d$ having $d+2$ 
elements. 
We always assume that the point configuration is full-dimensional, i.e. it 
affinely spans $\R^d$. 
Note that this is equivalent to assume the configuration to have corank one. 
In this case, we can simplify, and modify slightly, the notation for the 
entries of the volume vector. 
If $\A = \{ \bp_1, \ldots, \bp_{d+2}\}$, then we denote the volume vector 
of $\A$ as
\begin{equation}
    \label{eq:volumevector}
    w_\A = (w_1 , \ldots , w_{d+2}), \mbox{ where } 
    w_i \coloneq (-1)^{i+1} w_{1,\ldots,\hat{i},\ldots,d+2}.
\end{equation}
The change of sign allows to simplify the statement in the following lemma.

\begin{lem}[{\cite[Equation~(2)]{BS16a}}]\label{lem:encoding}
Let $\A=\{\bp_1,\ldots,\bp_{d+2}\}$ be a corank one point configuration. 
Then its volume vector \[w_\A = (w_1 , \ldots , w_{d+2})\] sums up to zero and 
encodes the unique linear relation in $\A$:
\[
    \sum_{i=1}^{d+2} w_i \bp_i = \bzero, 
    \qquad \text{and} \qquad 
    \sum_{i=1}^{d+2} w_i = 0.
\]
\end{lem}

One may think the equality of $\sum_{i=1}^{d+2} w_i = 0$ in the following way. 
The $w_i$'s which are positive are the normalized volumes of the full 
dimensional simplices in $\mathcal{T}_+$, while the negative $w_i$'s equal to 
(minus) the normalized volumes of the full-dimensional simplices in 
$\mathcal{T}_-$. The equality follows by noting that $\mathcal{T}_+$ and 
$\mathcal{T}_-$ are both triangulations of the same polytope. 
This is clarified by the following example.

\begin{ex}
    Let $\A$ be the point configuration given by the columns of the matrix 
    below.
    \[
        \begin{bmatrix}
            0 & 1 & 1 & 0 & 1\\
            0 & 1 & 0 & 1 & 1\\
            0 & 0 & 1 & 1 & 1\\
        \end{bmatrix}
    \]  
    Then, $\A$ is the set of the vertices of the polytope $P_\A$, depicted 
    below.
    \begin{figure}[H]
        \centering
        \tdplotsetmaincoords{65}{45}
        \begin{tikzpicture}[scale=1.5,tdplot_main_coords]
        \draw [thick,black!80,line join=round] (0,1,1)--(1,1,0);
        \fill [white!40,opacity=0.8] (0,0,0)--(1,1,0)--(1,1,1)--(0,1,1)--cycle;
        \draw [thick,black!80,line join=round] (0,0,0)--(1,1,0);
        \draw [thick,black!80,line join=round] (0,0,0)--(1,0,1);
        \draw [thick,black!80,line join=round] (0,0,0)--(0,1,1);
        \draw [thick,black!80,line join=round] (1,1,0)--(1,0,1)--(0,1,1);
        \draw [thick,black!80,line join=round] (1,1,1)--(1,1,0);
        \draw [thick,black!80,line join=round] (1,1,1)--(1,0,1);
        \draw [thick,black!80,line join=round] (1,1,1)--(0,1,1);
        \draw[black!80,fill=black!80] (0,0,0) circle (0.1em) node[below]{$\bp_1$};
        \draw[black!80,fill=black!80] (1,1,0) circle (0.1em) node[below]{$\bp_4$};
        \draw[black!80,fill=black!80] (0,1,1) circle (0.1em) node[above]{$\bp_2$};
        \draw[black!80,fill=black!80] (1,0,1) circle (0.1em) node[xshift=0cm, yshift=-.5cm]{$\bp_3$};
        \draw[black!80,fill=black!80] (1,1,1) circle (0.1em) node[above right]{$\bp_5$};
        \end{tikzpicture}
    \end{figure}
    \noindent The volume vector of $\A$ is $w_\A= (1,-1,-1,-1,2)$. 
    Note that the positive entries (1,2) in $w_\A$ corresponds exactly to the 
    normalized volumes of the two tetrahedra in $\mathcal{T}_+$, while the 
    negative entries (-1,-1,-1) corresponds exactly to (minus) the volumes of 
    the three tetrahedra in $\mathcal{T}_-$.
    
    \begin{center}
    \begin{minipage}{0.3\textwidth}
        \begin{figure}[H]
        \centering
        \tdplotsetmaincoords{65}{45}
        \begin{tikzpicture}[scale=1.5,tdplot_main_coords]
        \fill [white!40,opacity=0.8] (0,0,0)--(1,1,0)--(0,1,1)--cycle;
        \draw [thick,black!80,line join=round] (0,0,0)--(1,1,0);
        \draw [thick,black!80,line join=round] (0,0,0)--(1,0,1);
        \draw [thick,black!80,line join=round] (0,0,0)--(0,1,1);
        \draw [thick,black!80,line join=round] (1,1,0)--(1,0,1)--(0,1,1)--cycle;
        \draw[black!80,fill=black!80] (0,0,0) circle (0.1em) node[below]{};
        \draw[black!80,fill=black!80] (1,1,0) circle (0.1em) node[below]{};
        \draw[black!80,fill=black!80] (0,1,1) circle (0.1em) node[above]{};
        \draw[black!80,fill=black!80] (1,0,1) circle (0.1em) node[left]{};
        \begin{scope}[shift={(.12,.12,.12)}]
        \draw [thick,black!80,line join=round] (1,1,0)--(0,1,1);
        \fill [white!40,opacity=0.8] (1,1,0)--(1,1,1)--(0,1,1)--(1,0,1)--cycle;
        \draw [thick,black!80,line join=round] (1,1,0)--(1,0,1)--(0,1,1);
        \draw [thick,black!80,line join=round] (1,1,1)--(1,1,0);
        \draw [thick,black!80,line join=round] (1,1,1)--(1,0,1);
        \draw [thick,black!80,line join=round] (1,1,1)--(0,1,1);
        \draw[black!80,fill=black!80] (1,1,0) circle (0.1em) node[below]{};
        \draw[black!80,fill=black!80] (0,1,1) circle (0.1em) node[above]{};
        \draw[black!80,fill=black!80] (1,0,1) circle (0.1em) node[left]{};
        \draw[black!80,fill=black!80] (1,1,1) circle (0.1em) node[above right]{};
        \end{scope}
        \end{tikzpicture}
        \caption*{$\mathcal{T}_+$}
        \end{figure}
    \end{minipage}
    ~
    \begin{minipage}{0.3\textwidth}
        \begin{figure}[H]
        \centering
        \tdplotsetmaincoords{65}{45}
        \begin{tikzpicture}[scale=1.5,tdplot_main_coords]
        \begin{scope}[shift={(.4,.1,.15)}]
        \draw [thick,black!80,line join=round] (1,1,0)--(0,1,1);
        \fill [white!40,opacity=0.8] (0,0,0)--(0,1,1)--(1,1,1)--(1,1,0)--cycle;
        \draw [thick,black!80,line join=round] (0,0,0)--(1,1,1);
        \draw [thick,black!80,line join=round] (0,0,0)--(0,1,1);
        \draw [thick,black!80,line join=round] (0,0,0)--(1,1,0);
        \draw [thick,black!80,line join=round] (1,1,0)--(1,1,1)--(0,1,1);
        \draw[black!80,fill=black!80] (0,0,0) circle (0.1em) node[below]{};
        \draw[black!80,fill=black!80] (1,1,0) circle (0.1em) node[above]{};
        \draw[black!80,fill=black!80] (0,1,1) circle (0.1em) node[left]{};
        \draw[black!80,fill=black!80] (1,1,1) circle (0.1em) node[left]{};
        \end{scope}
        \fill [white!40,opacity=0.8] (0,0,0)--(1,1,1)--(0,1,1)--cycle;
        \draw [thick,black!80,line join=round] (0,0,0)--(1,1,1);
        \draw [thick,black!80,line join=round] (0,0,0)--(1,0,1);
        \draw [thick,black!80,line join=round] (0,0,0)--(0,1,1);
        \draw [thick,black!80,line join=round] (1,1,1)--(1,0,1)--(0,1,1)--cycle;
        \draw[black!80,fill=black!80] (0,0,0) circle (0.1em) node[below]{};
        \draw[black!80,fill=black!80] (0,1,1) circle (0.1em) node[above]{};
        \draw[black!80,fill=black!80] (1,0,1) circle (0.1em) node[left]{};
        \draw[black!80,fill=black!80] (1,1,1) circle (0.1em) node[left]{};
        \begin{scope}[shift={(.3,-.1,-.1)}]
        \draw [thick,black!80,line join=round] (0,0,0)--(1,1,1);
        \fill [white!40,opacity=0.8] (0,0,0)--(1,0,1)--(1,1,1)--(1,1,0)--cycle;
        \draw [thick,black!80,line join=round] (0,0,0)--(1,0,1);
        \draw [thick,black!80,line join=round] (0,0,0)--(1,1,0);
        \draw [thick,black!80,line join=round] (1,1,1)--(1,0,1)--(1,1,0)--cycle;
        \draw[black!80,fill=black!80] (0,0,0) circle (0.1em) node[below]{};
        \draw[black!80,fill=black!80] (1,1,0) circle (0.1em) node[above]{};
        \draw[black!80,fill=black!80] (1,0,1) circle (0.1em) node[left]{};
        \draw[black!80,fill=black!80] (1,1,1) circle (0.1em) node[left]{};
        \end{scope}
        \end{tikzpicture}
        \caption*{$\mathcal{T}_-$}
        \end{figure}
    \end{minipage}
    \end{center}

    \noindent In particular the entries of $w_\A$ sum to zero. 
    We can furthermore check that they encode the unique affine linear 
    relation in $\A$, indeed
    \[
    \bp_1 - \bp_2 - \bp_3 - \bp_4 + 2 \bp_5 = \bzero.
    \]
\end{ex}

\section{An implementable proof for the Lagarias--Ziegler Theorem}
\label{sec:LZproof}

In this section we give an alternative algorithmic proof to
Theorem~\ref{thm:LZ}.
Such algorithm will be then implemented for a complete enumeration of lattice 
polytopes with ``reasonably small'' volume and dimension, which is described in
the following sections.

The original proof given in \cite{LZ91} is divided in two parts, one proving
the result for simplices, another extending it to polytopes. The ``simplicial''
part of the result is easily deduced by putting the matrix of the vertices
of a simplex in a normal form.
This part of the proof, as given in \cite{LZ91}, can be easily implemented,
so we can use it as the first step of the algorithm, adding only some small
improvements (see Algorithm~\ref{alg:sim}). For the convenience of the reader we
quickly sketch the theoretical argument used.

\begin{lem}\label{lem:LZsimplices}
    There are finitely many equivalence classes of $d$-dimensional lattice 
    simplices having volume lower than a constant $k$.
\end{lem}
\begin{proof}
    Let $S$ be the $d$-dimensional lattice simplex with vertices
    $\bv_0$, $\bv_1$,$\ldots$, $\bv_{d+1}$.
    We can suppose $v_0$ to be the origin of the lattice.
    In this way $\vol(S)=|\det (M)|$, where $M$ is the $d \times d$ matrix
    whose columns are the vertices $\bv_1,\ldots,\bv_{d+1}$.
    We now take $M$ to the upper-triangular form (Hermite normal form)
    \[
        M' \coloneq
        \begin{pmatrix}
        v_{1,1}	&	v_{1,2}	&	\cdots	&	v_{1,d}	\\
    	    	&	v_{2,2}	&	\cdots	&	v_{2,d}	\\
        		&			&	\ddots	&	\vdots	\\
    	    	&			&			&	v_{d,d}	\\
         \end{pmatrix},
    \]
    where on the $j$-th column $ 0 \leq a_{i,j} < a_{j,j}$ for $i=1,\ldots,j$,
    and $a_{i,j}=0$ for $i=j+1,\ldots,d$. 
    Since the simplex $S'$ whose vertices are the origin and the columns of
    $M'$ is affinely equivalent to $S$, $\Vol(S)=\prod_{i=1}^{d}a_{i,i}$.
    Since all the entries of $M'$ are positive, there are finitely many
    possible values for all the entries $a_{i,i}$, and consequently, for all
    the entries of $M'$.
\end{proof}

From now on, our proof diverges from the original one.
In particular, we now focus our attention on the case of $d$-dimensional 
polytopes having $d+2$ vertices. We prove the result for this special case 
using the theory developed in the previous section and then we deduce the 
general case as a corollary.

\begin{prop}\label{prop:LZd+2}
    There are finitely many equivalence classes of $d$-dimensional lattice
    polytopes having $d+2$ vertices and volume lower than a constant $K$.
\end{prop}
\begin{proof}
    Let $P$ be a lattice polytopes with $d+2$ vertices such that $\Vol(P) < K$.
    We call $\A$ the point configuration given by the vertices of $P=P_\A$, it
    has corank one.
    Let $w_\A$ be the volume vector of $\A$.
    Since the volume of $P$ is bounded by $K$, the sum of the positive entries
    of $w_\A$ is $K$ at most.
    Similarly, the sum of the negative entries in $w_\A$ is $-K$ at least.
    In particular there are finitely many possible volume vectors vectors that
    can be the volume vector of $\A$.
    By Lemma~\ref{lem:encoding} the previous statement means there are
    finitely many possible dependence relations which can be (up to
    multiplication by a scalar factor) the only dependence relation on $\A$.
    This proves that, if $S$ is any $d$-dimensional lattice simplex with
    $\Vol(S) < K$, then that the set 
    \[
        \left\{ \bp \in \Z^d : \Vol(\conv(S \cup \{\bp\})) \leq K \right\}
    \]
    is finite. 
    Finally, we note that $P$ is completely determined by the choice
    of $d+1$ affinely independent and ordered vertices, plus the unique linear
    relation among its $d+2$ vertices.
    The convex hull of $d+1$ affinely independent vertices of $P$ is a 
    $d$-dimensional simplex having normalized volume strictly smaller than $K$.
    By Lemma~\ref{lem:LZsimplices}, there are finitely many such simplices.
\end{proof}

From this we are able to prove Lagarias--Ziegler Theorem.

\begin{proof}[Proof of Theorem~\ref{thm:LZ}]
    A $d$-dimensional lattice polytope $P$ of volume $\Vol(P) \leq K$ has
    at most $K+d$ lattice points. This can be immediately deduced using some 
    basic Ehrhart Theory (see Proposition~\ref{prop:h^*_basic}).
    In particular, its number of vertices $n$ is at most $K+d$.
    Suppose the vertices of $P$ are ordered such that
    $\bv_0,\ldots,\bv_d$ are affinely independent.
    Then
    \[
        P= \conv \left( \bigcup_{i=d+2}^n P_i \right),
    \]
    where $P_i \coloneq \conv (\bv_0,\ldots,\bv_d,\bv_i)$ with $d+2 \leq i
    \leq n$.
    For each $i$, $P_i$ is a $d$-dimensional polytope with $d+2$ vertices,
    while $S\coloneq \conv(\bv_0,\ldots,\bv_d)$ is a $d$-dimensional simplex.
    Since $\Vol(S) \leq \Vol(P_i) \leq \Vol(P) \leq K$, we conclude by 
    Lemma~\ref{lem:LZsimplices} and Proposition~\ref{prop:LZd+2}.
\end{proof}

\section{Implementation}
\label{sec:implementation}

In this section we describe the implementation of the results described in the 
previous section. Note that both the proofs of Theorem~\ref{thm:LZ} for 
simplices and for polytopes admit a straightforward algorithmic implementation.
Anyway, in order to have feasible running times, we are going to optimize the
algorithms with some careful tweaking.
Results of such implementation will be described in Section~\ref{sec:results}.

We denote by $\Pol_V^d$ and $\Sim_V^d$ the sets of $d$-dimensional polytopes
and simplices having normalized volume $V$.
Note that, as usual, these sets are considered up to unimodular equivalence.
Computationally speaking, this is not a problem: each polytope can be indeed 
put in a normal form, and $\Pol_V^d$ can be thought as the set of these forms.

Algorithm~\ref{alg:sim} fully enumerates all the elements of $\Sim_V^d$.
To speed things up, we can ``recycle'' the enumeration of $\Sim_V^{d-1}$, being
the case $\Sim_V^1$ trivial.

\begin{algorithm}[htp]
    \SetKwInOut{Input}{input}\SetKwInOut{Output}{output}
    \Input{$\Sim_V^{d-1}$}
    \Output{$\Sim_V^d$}
    \BlankLine
    $\Sim_V^d \longleftarrow \varnothing$\;   
    \For{$v$ such that $v|V$}{
        $p_d \longleftarrow \tfrac{V}{v}$\;
        \For{$F \in \Sim_{V}^{d-1}$}{
            \For{$p_1,\ldots,p_{d-1} \in [0,p_d-1]^{d-1}$}{
                $\bp \longleftarrow (p_1,\ldots,p_{d-1},p_d)$\;
                $S \longleftarrow \conv( (F \times 0) \cup \bp )$\;
                $\Sim_V^d \longleftarrow \Sim_V^d \cup \{S\}$\;
            }
        }
    }
    \caption{
        \label{alg:sim}
        The algorithm for the enumeration of all the elements of $\Sim_V^d$
    }
\end{algorithm}

We now discuss the implementation for the complete enumeration of the elements 
of $\Pol_V^d$.
We denote by $\Pol_{\leq K}^d$ the set of all $d$-dimensional lattice polytopes
of normalized volume at most $K$, i.e.
$\Pol_{\leq K}^d \coloneq \bigcup_{V = 1}^K \Pol_V^d$. Similarly, we set
$\Sim_{\leq K}^d \coloneq \bigcup_{V = 1}^K \Sim_V^d$.
The algorithm works as follows. 
The simplices of $\Sim_{\leq K}^d$ are used as starting objects for the
enumeration.
The possible volume vectors of point configuration of cardinality $d+2$ are 
then calculated and used to iteratively add new vertices to the simplex.
This is possible because the volume vector of a point configurations with 
$d+2$ points encodes the unique affine dependence among them 
(Lemma~\ref{lem:encoding}).
In order to optimize the implementation, the volume vectors have to be chosen 
carefully.
We use the sign-changed definition of the volume vector given in 
\eqref{eq:volumevector}. We denote by $\W_K^d$ the set
{\small
\[
    \W_K^d \coloneq \left\{ (w_1,\ldots,w_{d+2}) \in [-K,K]^{d+2} : 
    \sum_{w_i > 0} w_i = - \sum_{w_i < 0} w_i < K \right\} \cap \Z^{d+2}.
\]
}
It contains all the possible volume vectors of point configurations of $d+2$ 
points whose convex hull has normalised volume at most $K$.
Once a simplex $S$ is fixed, we can add new vertices to it using only the 
volume vectors having $\Vol(S)$ as the first entry.
To make the computation even faster one can assume that $\Vol(S)$ is, in
absolute value, the highest entry in the volume vector. 
This means that each polytope would be built starting only from the biggest 
simplex it contains. 
Using the volume vectors to determine the growing step of the classification 
algorithm also provides a handy way to deal with symmetries.
Algorithm~\ref{alg:pol} requires $\Sim_{\leq K}^d$ and $\W_K^d$ as inputs, 
and returns $\Pol_{\leq K}^d$ as output. $\Sim_{\leq K}^d$ is obtained by
iterating Algorithm~\ref{alg:sim} for values of $V$ ranging from 1 to $K$, 
while $\W_K^d$ can be trivially computed.
Given a simplex $S=\conv(\bv_0,\ldots,\bv_{d}) \in \Sim_{\leq K}^d$ and a 
volume vector $w=(w_1,\ldots,w_{d+2}) \in \W_K^d$ such that $w_{d+2}=\Vol(S)$,
we define $\bp_{S,w}$ to be the point of $\R^d$ such that $w$ is the volume
vector of the point configuration $\{\bv_0,\ldots,\bv_{d},\bp_{S,w}\}$. 
Thanks to Lemma~\ref{lem:encoding}, $\bp_{S,w}$ is uniquely determined, indeed 
\[
    \bp_{S,w} = \frac{- \sum_{i=0}^{d} w_{i+1} \bv_i  }{\Vol(S)}. 
\]
Note that, in general, $\bp_{S,w}$ is not a lattice point.
At every iteration of Algorithm~\ref{alg:pol}, all the points $\bp_{S,w}$ that
are lattice points are stored in a temporary variable $\X_S$.
After that, the elements of $\X_S$ are used to ``grow'' $S$ in all the possible
ways.
This is done in the second part of the main loop.
A variable called $s$ is used to count how many iterations the growing process 
needs.
One can think to $s$ as the variable counting the $\emph{size}$ of the lattice 
polytopes, i.e. their number of lattice points.
In particular, at every iteration over $s$, only the lattice polytopes of size
$s$ will be processed and ``grown'' by adding one lattice point.
Consider that this process ramifies and becomes slower, indeed starting
from a single simplex, adding different points obviously generates different
polytopes. 
In order to minimize the number of iterations we use some Ehrhart Theory,
which guarantees a simple structure for the polytopes
for which the number of lattice points is maximal with respect to the volume.
This is done via Lemma~\ref{lem:degree1}. In order to state it correctly, we 
need some definitions.

Given a $d$-dimensional lattice polytope $P \subset \R^d$, we define the 
\emph{lattice pyramid} $\Pyr(P)$ as the $(d+1)$-dimensional polytope
\[
    \Pyr(P) \coloneq \conv(P \times \{0\} \cup \{(0,\ldots,0,1)\})
    \subset \R^{d+1}.
\]
Moreover, we say that a $d$-dimensional lattice polytope $P$ is an 
\emph{exceptional simplex} if $P$ can be obtained via the $(d-2)$-fold 
iterations of the lattice pyramid construction over the second dilation of a 
unimodular simplex, that is,
\[
    P \cong \Pyr(\cdots(\Pyr(\conv((0,0),(2,0),(0,2))))\cdots).
\]
We say that a $d$-dimensional lattice polytope $P \subseteq \R^d$ is a \emph{Lawrence prism}
with \emph{heights} $a_0, \ldots, a_{d-1}$ if there exist nonnegative integers 
$a_0, \ldots , a_{d-1}$ such that
\[
    P \cong \conv (\{ \bzero , a_0 \be_d , \be_1 , \be_1 + a_1 \be_d , \ldots ,
    \be_{d-1}, \be_{d-1} + a_{d-1} \be_d \}),
\]
where $\be_1,\ldots,\be_d$ denote the standard basis of $\R^d$ 
\begin{lem}
    \label{lem:degree1}
    Let $P$ be a $d$-dimensional lattice polytope.
    Then $|P \cap \Z^d| \leq d + \Vol(P)$, with equality if and only if
    $P$ is either an exceptional simplex or a Lawrence prism.
\end{lem}
\begin{proof}
    The first part of the lemma is also known as Blichfeldt's Theorem 
    \cite{Bli14}, but in this case it can be easily deduced with some basic
    Ehrhart Theory (see Proposition~\ref{prop:h^*_basic}). 
    A polytope for which the equality $|P \cap \Z^d| = d + \Vol(P)$ 
    is attained, must have $\h$-polynomial $\h_P(t) = 1 + (\Vol(P)-1) t$, 
    which has degree one. 
    The rest of the statement is exactly the characterization of lattice
    polytopes having $\h$-polynomial of degree one by Batyrev--Nill
    \cite{BN07}.
\end{proof}

Thanks to this lemma, we know that the variable $s$ can range between
$|S \cap \Z^d|$ and $K+d-1$. At every iteration over $s$  the algorithm selects
all the lattice polytopes of size $s$, it grows them into larger ones, which
are then stored into a set $\Q_S$. 
The union of all the $\Q_S$ for all the simplices $S \in \Sim_{\leq K}^d$ will
be the complete list of $d$-dimensional lattice polytopes having volume at most
$K$, except possibly some Lawrence prisms. Those are easy to classify and can
be added ``manually''  as the last step of the algorithm.

\begin{algorithm}[htp]
    \SetKwInOut{Input}{input}\SetKwInOut{Output}{output}
    \Input{$\Sim_{\leq K}^d$, $\W_K^d$}
    \Output{$\Pol_{\leq K}^d$}
    \BlankLine
    $\Pol_{\leq K}^d \longleftarrow \Sim_{\leq K}^d$\;   
    \For{$S \in \Sim_{\leq K}^d$}{
        $\X_S \longleftarrow \varnothing$\;
        $\Q_S \longleftarrow \{S\}$\;
        \For{$w=(w_1,\ldots,w_{d+2}) \in \W_K^d$ such that $w_{d+2}=\Vol(S)$}{
            \If{$\bp_{S,w} \in \Z^d$}{
                $\X_S \longleftarrow \X_S \cup \{ \bp_{S,w}\}$\;
            }
        }
        \For{$s \in [|S \cap \Z^d|,K+d-1]$}{
            \For{$P \in \Q_S$ such that $|P \cap \Z^d| = s$}{
                \For{$\bp \in \X_S$}{
                    $Q \longleftarrow \conv(P \cup \{\bp\})$\;
                    \If{$\Vol(Q) \leq K$}{
                        $\Q_S \longleftarrow \Q_S \cup \{Q\}$\;
                    }
                }
            }    
        }
        $\Pol_{\leq K}^d \longleftarrow \Pol_{\leq K}^d \cup \Q_S$\;
    }
    $\Pol_{\leq K}^d \longleftarrow \Pol_{\leq K}^d \cup \{ P : \Vol(P)\leq K
    \mbox{ and $P$ is Lawrence prism} \}$\;
    \caption{
        \label{alg:pol}
        The algorithm for the complete enumeration of the elements of
        $\Pol_{\leq K}^d$
    }
\end{algorithm}

Beside the optimizations discussed in this section, other small improvements
above have been added to the actual implementations of Algorithms~\ref{alg:sim}
and \ref{alg:pol}.
Such expedients are obvious verifications, such as not trying to add point to 
lattice polytopes of volume $K$ or not trying to add the same point twice, 
and are not reported in the pseudocode, in order to keep it essential.

\section{Results and comparison with existing classifications}
\label{sec:results}

In this section we discuss the results of the classification, and we compare it
with other existing ones. 
Algorithms~\ref{alg:sim} and \ref{alg:pol} have been implemented in {\sc Magma}
\cite{BCP972} on Intel Core i7-2600 CPU 3.40GHz. 
The total running time for all the classifications was roughly one year in 
total.

Note that the implementation can be easily parallelized and run on large
clusters, but this was beyond the resources of the author and the aims of the 
paper.
The implementation has therefore been performed in low dimensions (up to six)
and finding a compromise between a large enough volume and a fast enough 
running time.
The classifications are therefore far from any kind of computational 
limit, and, if there will be request, they can easily be pushed forward.

The two-dimensional case of the classification is not of particular interest, 
as lattice polygons up to 30 interior lattice points have been classified in
\cite{Cas12} and the \emph{hollow} ones, i.e. the ones without interior lattice 
points, can be easily described. Indeed, a hollow lattice polygon can either be
the convex hull of two lattice segments in two consecutive lines, or the
exceptional simplex. 
For completeness, and for making comparisons, this case has been computed 
anyway, but the computation has been stopped after a few hours.

Specifically, we fully enumerate the elements of $\Sim_{\leq K}$ and 
$\Pol_{\leq K}^d$ for the following couples
$d$ and $K$:
\begin{itemize}
    \item $d = 2$ and $K = 50$,
    \item $d = 3$ and $K = 36$,
    \item $d = 4$ and $K = 24$,
    \item $d = 5$ and $K = 20$,
    \item $d = 6$ and $K = 16$.
\end{itemize}

Having in mind applications to Ehrhart Theory, the enumeration of 
$\Sim_{\leq K}^3$ has been performed for $K=1000$ and discussed in 
Section~\ref{sec:3-simplices}.

We report here the outcome of the implementation of Algorithm~\ref{alg:pol}, 
but we discuss applications and implications in the following sections.

\begin{thm}
    \label{thm:result}
    Up to unimodular equivalence there are 
    \begin{itemize}
        \item $408 \, 788$ two-dimensional lattice polytopes having
        	  volume at most 50.
        \item $6 \, 064 \, 034$ three-dimensional lattice polytopes having
	          volume at most 36.
        \item $989 \, 694$ four-dimensional lattice polytopes having
	          volume at most 24.
        \item $433 \, 273$ five-dimensional lattice polytopes having
	          volume at most 20.
        \item $117 \, 084$ six-dimensional lattice polytopes having
	          volume at most 16.
    \end{itemize}
    Their distribution according to their volume, can be read from the
    tables in Appendix~\ref{app:tables}.
\end{thm}

The resulting database is available at \url{https://github.com/gabrieleballetti/small-lattice-polytopes}.
An immediate application of a classification result such this, is to
double check existing classifications.
For example Theorem~\ref{thm:result} has already been used to correct mistakes
in a recent classification result by Hibi and Tsuchiya.
In \cite{HT17} they give a characterization for all the lattice polytopes of 
any dimension, having normalized volume lower than or equal to four.
By comparing their results with the classification above it turned out that 
some polytopes in dimension four and five were missing from their lists.
The current version has been corrected with the missing polytopes.

In dimension three our classification has large intersections with several
existing ones. In particular we checked that it agrees with:
\begin{enumerate}
    \item the classification of three-dimensional lattice polytopes with
          one interior lattice point \cite{Kas10};
    \item the classification of three-dimensional lattice polytopes with
          two interior lattice points \cite{BK16};
    \item the classification of three-dimensional lattice polytopes with
          width larger than one and having up to eleven lattice points
          \cite{BS17}.
\end{enumerate}
Additionally, our classification fully contains:
\begin{enumerate}
    \setcounter{enumi}{3}
    \item the 12 hollow three-dimensional lattice polytopes which are maximal 
          up to inclusion classified in \cite{AWW11,AKW17};
    \item the classification of three-dimensional smooth polytopes (see 
          Section~\ref{sec:smooth} for a definition), having up to 16 lattice 
          points, which have been classified in several steps
          \cite{BHH15,Lor10,Lun13}.
\end{enumerate}

The fact that our classification contains the 12 hollow three-dimensional 
lattice polytopes which are maximal up to inclusion of \cite{AWW11,AKW17} is not
surprising.
On the contrary, the classification of three dimensional polytopes was pushed up
to normalized volume 36 in order to compare the two classifications.
Indeed the largest hollow maximal three-dimensional polytope has volume 36
(this is interesting on its own, as it seems to suggest a ``hollow case'' of 
Conjecture~\ref{conj:volume_bound}). 
Anyway, we remark that our classification does not give an independent proof
of this fact.
On the other hand, it is remarkable how Algorithm~\ref{alg:pol} seems to be
faster than the algorithms to classify smooth polytopes used in
\cite{BHH15,Lor10,Lun13}. 

\section{Smooth polytopes}
\label{sec:smooth}

A natural property in the study of lattice polytopes (especially when it is 
motivated by toric geometry) is the smoothness.
A lattice polytope $P$ in $\R^d$ is called \emph{smooth} if it is simple and 
if its primitive edge directions at every vertex form a basis of $\Z^d$.
Sometimes, smooth polytopes are also called \emph{Delzant}.
The word ``smooth'' comes indeed from the toric varieties realm: a lattice 
polytope is smooth if and only if the associated projective toric variety is 
smooth (see \cite[Section~2.1]{Ful93}).

The most important open problem regarding smooth polytopes is the so called
Oda's conjecture, for which we need to introduce the notion of Integer 
Decomposition Property. 
We say that a $d$-dimensional lattice polytope $P$ is \emph{IDP}, or has the
Integer Decomposition Property, if for every integer $n \geq 1$ and every
lattice point $\bp \in nP \cap \Z^d$ there are lattice points $\bp_1, 
\ldots, \bp_n \in P \cap \Z^d$ such that $\bp = \bp_1 + \cdots + \bp_n$.
Polytopes having this property are often referred to as \emph{integrally 
closed}, but one should not confuse them with \emph{normal} polytopes, which
are those which are IDP when considered as lattice polytopes with respect to
the lattice affinely spanned by their lattice points.
Being IDP is also a very natural property (one can think of it as a discrete
counterpart of convexity), which is of interest in algebraic geometry, 
combinatorics, commutative algebra and optimizations.

In the nineties Oda \cite{Oda08} formulated several problem on Minkowski sums
of lattice polytopes.
One of them is nowadays known in the following form as Oda's Conjecture.

\begin{conj}[Oda's conjecture]
    \label{conj:oda}
    Every smooth lattice polytope is IDP.
\end{conj}

This is seemingly innocuous statement is actually open even in low dimensions, 
and in many stronger forms.

In \cite{BHH15} Bogart et al. prove that for every nonnegative integers $d$ and
$n$ there are, modulo unimodular equivalence, there are only finitely many 
$d$-dimensional smooth polytopes with $n$ lattice points.
This can be seen as an evidence for the validity of Conjecture~\ref{conj:oda}, 
as is would follow from it as a corollary (it is indeed easy to verify that 
there are finitely many IDP polytopes once dimension and number of lattice
points are fixed).
As an application of their result they classify smooth three-dimensional
polytopes having up to 12 lattice points (see also \cite{Lor10}). 
This was later extended by Lundman \cite{Lun13} who classified all the lattice
polytopes having up to 16 lattice points.

\begin{thm}[{\cite[Theorem~1]{Lun13}}]
    Up to unimodular equivalence there exist exactly 103 smooth 
    three-dimensional lattice polytopes $P \subseteq \R^3$ such that 
    $|P \cap \Z^3| \leq 16$.
\end{thm}

The largest polytope in Lundman's classification has normalized volume 23. 
As a consequence the following result enlarges the current census of
``small'' three-dimensional polytopes.

\begin{thm} 
    \label{thm:smooth-3p}
    Up to unimodular equivalence there exist exactly $1 \, 588$ smooth
    three-dimensional lattice polytopes $P \subseteq \R^3$ such that 
    $\Vol(P) \le 36$. 
    The 103 polytopes having at most 16 lattice points are a subset of them.
    The distribution of smooth three-dimensional polytopes by their volume
    is summarized in Table~\ref{tab:smooth-3p}.
\end{thm}

This highlights how Algorithm~\ref{alg:pol} seems to be more efficient
than the ones used to classify smooth polytopes in \cite{BHH15,Lor10,Lun13},
althought it is not shaped to deal with smooth polytopes.
Moreover, Algorithm~\ref{alg:pol} can be freely used in higher dimension.
In particular we can easily obtain results analogous to 
Theorem~\ref{thm:smooth-3p} results up to dimension six.

\begin{thm}
    \label{thm:smooth}
    Up to unimodular equivalence there are 
    \begin{itemize}
        \item $1 \, 530$ two-dimensional smooth polytopes having normalized
              volume at most 50;
        \item $1 \, 588$ three-dimensional smooth polytopes having normalized 
              volume at most 36;
        \item $738$ four-dimensional smooth polytopes having normalized 
              volume at most 24;
        \item $412$ five-dimensional smooth polytopes having normalized 
              volume at most 20;
        \item $127$ six-dimensional smooth polytopes having normalized 
              volume at most 16.
    \end{itemize}
    The distribution of the classified smooth polytopes by their volume
    is summarized in Appendix~\ref{app:smooth}.
\end{thm}

Note that, in dimension two, volume and number of lattice points are
strictly correlated.
It is easy to prove that, for a two-dimensional polytope $P$ 
\[
   |P \cap \Z^2| - 2 \leq \Vol(P) \leq 2|P \cap \Z^2|-5. 
\]
One can see as a consequence of some basic results of Ehrhart Theory, 
developed in the following section.
As a consequence, the classification of two-dimensional smooth polytopes
contains all those having up to 27 lattice points.
This extends \cite[Theorem~32]{BHH15}.

\begin{cor}
    Up to unimodular equivalence there are exactly $458$ two-dimensional smooth
    polytopes having up to 27 lattice points.
\end{cor}

Conjecture~\ref{conj:oda} can now be easily verified on the classified 
polytopes.

\begin{thm}
    Conjecture~\ref{conj:oda} holds for all the smooth polytopes of 
    Theorem~\ref{thm:smooth}.
\end{thm}

By observing Tables~\ref{tab:smooth-2p}-\ref{tab:smooth-6p} in Appendix~\ref{app:smooth}, one can notice 
that, in each dimension $d \leq 6$, there are only two smooth polytopes of 
normalized volume lower than or equal to $d$.
They are the \emph{unimodular simplex} $\D_d$, defined as the convex hull of 
the origin and the standard basis, which has volume one, and the prism
$\D_{d-1} \times \D_{1}$, which has volume $d$.
We now verify that this is always the case, for all $d$.
This is indeed a consequence of the combinatorics of simple polytopes.

We are going to use two classical results for simple polytopes. The first is 
Barnette's Lower Bound Theorem for simple polytopes.

\begin{thm}[{\cite[Theorems~1-2]{Bar71}}]
    \label{thm:barnette}
    Let $P$ be a $d$-dimensional simple polytope.
    Denote by $f_0$ and $f_{d-1}$ the number of vertices and facets of
    $P$, respectively.
    Then,
    \[
        f_0 \geq (d-1)f_{d-1} - (d+1)(d-2).
    \]
    If $ d\geq 4$, the inequality above is attained with equality only if $P$
    of obtained from a simplex via progressive truncations of vertices.
\end{thm}

After that we are going to use a description of simple polytopes with $d+2$
facets, that can be found (in a dual version for simplicial polytopes) in
Gr{\"u}nbaum's textbook.

\begin{thm}[{\cite[Theorem~6.1.1]{Gru03}}]
    \label{thm:grunbaum}
    There exist $\lfloor \frac{1}{2}d \rfloor$ combinatorial types of 
    $d$-dimensional simple polytopes with $d+2$ facets.
    Those are exactly the products $\D_{d-i} \times \D_{i}$ for 
    $i=1,\ldots,\lfloor \frac{1}{2}d \rfloor$.
\end{thm}
    
As a corollary of these result, we describe the combinatorics of simple
lattice polytopes having few lattice points.

\begin{lem}
    \label{lem:few_vertices}
    Let $P$ be a simple $d$-dimensional lattice polytope having at most $3d-4$ 
    lattice points.
    Then $P$ is either combinatorially equivalent to the simplex $\D_d$, or to
    the prism $\D_{1} \times \D_{d-1}$.
\end{lem}
\begin{proof}
    $P$ has at most $3d-4$ vertices.
    By plugging this number into the inequality of Theorem~\ref{thm:barnette}, 
    we get an upper bound for the number of facets $f_{d-1}$ that $P$ can have:
    \[
        f_{d-1} \leq d+2.
    \]
    Since $f_{d-1} = d+1$ if and only if $P$ is a simplex, we can focus on the 
    case $f_{d-1} = d+2$. 
    By Theorem~\ref{thm:grunbaum} $P$ is combinatorially equivalent to the 
    product $\D_i \times \D_{d-i}$, for some 
    $1 \leq i \leq \lfloor \frac{1}{2} d \rfloor$.
    In particular, $P$ has $f_0 = (i+1)(d-i+1)$ vertices. 
    Fixing $d$, this quantity is growing in $i$, so the inequality
    \[
        f_0 = (i+1)(d-i+1) \leq 3d-4
    \]
    is satisfied if and only if $i=1$.
\end{proof}

Now, by adding the constraint of being a smooth polytope, we can get a simple
description of smooth polytopes having few lattice points.

\begin{thm}
    \label{thm:smooth_few_points}
    Let $P$ be a smooth $d$-dimensional lattice polytope having at most $3d-4$
    lattice points.
    Then $P$ is a either the unimodular simplex $\D_d$ or a Lawrence prism
    with heights $a_0,\ldots,a_{d-1}$ with $a_i \geq 1$ for all $i$,  and 
    $\sum_{i=0}^{d-1} a_i = 2d-4$.
\end{thm}
\begin{proof}
    If $P$ is a simplex, the statement is trivial. 
    By Lemma~\ref{lem:few_vertices}, $P$ is combinatorially equivalent to 
    $\D_1 \times \D_{d-1}$, in particular $P$ has two facets $F$ and $F'$ which 
    are $(d-1)$-dimensional simplices.
    Since faces of smooth polytopes are smooth, $F$ and $F'$ are dilations of 
    $\D_{d-1}$. 
    The $t$-th dilation of a $(d-1)$-dimensional simplex has 
    $\binom{d+t-1}{d-1}$ lattice points (see e.g. \cite[Theorem~2.2]{BR15}),
    which is lower than or equal to $3d-4$ only for $t=1$.
    This proves that both $F$ and $F'$ are unimodular equivalent to $\D_{d-1}$.
    Since $P$ is smooth we can assume that
    \[
        F = \D_{d-1} = \conv(\bzero,\be_1,\ldots,\be_{d-1}),
    \]
    and that $\be_d$ is a lattice point of $P$.
    Let $E_0,\ldots,E_{d-1}$ be the edges of $P$ which are not in $F$ nor in
    $F'$, labeled such that $\be_i$ is one vertex of $E_i$ for 
    $i=1,\ldots,d-1$, while $\bzero$ and $\be_d$ are in $E_0$ ($\be_d$ might
    not be a vertex).
    Let moreover $\bp_i$ be the first lattice point met traveling from $\be_i$
    along $E_i$, for $i=0,\ldots,d-1$, so that $\bp_0=\be_d$.
    The statement follows by proving that 
    \[
        \conv(\bzero,\be_1,\ldots,\be_{d-1},\bp_0,\ldots,\bp_{d-1}) 
        = F \times \D_1.
    \]
    By the smoothness assumption, the simplex $\conv(F \cup \bp_i)$ is
    unimodular for all $i$.
    This proves that all the $\bp_i$'s are at height one, i.e
    \[
        \tilde{F} \coloneq \conv(\bp_0,\ldots,\bp_{d-1}) = P \cap \{x_d = 1\}.
    \]
    The combinatorics of $P$ implies that $\tilde{F}$ equals to $tF \times {1}$
    for a dilation factor $t$, but by the same argument as above, $t=1$.
\end{proof}

Having a Lawrence prism structure is very restricting.
Lattice points and volume of a Lawrence prisms $P$ are linked by the formula 
$|P \cap \Z^d| = d + \Vol(P)$ (see Lemma~\ref{lem:degree1}).

\begin{cor}
    In dimension $d$ the only smooth polytopes having normalized volume
    at most $d$ are the unimodular simplex $\D_d$ and the prism 
    $\D_{d-1} \times [0,1]$.
    They have normalized volume 1 and $d$, respectively.
\end{cor}

Lawrence prisms have a very restrictive geometry.
It is easy to show (for example using pushing or pulling triangulations)
that they have a \emph{quadratic} triangulation, i.e. a triangulation which
is regular unimodular and flag. 
We refer to \cite{HPPS14} for definitions and terminology about triangulations.
The existence of quadratic triangulations for a smooth polytope is a central question
in toric geometry as it implies that the associated projective toric variety has
a defining ideal generated by quadrics (see \cite{Stu96}). This problem and several of its variations
are sometimes known as B\"{o}gvad Conjecture.

\begin{cor}
    Let $P$ be a $d$-dimensional smooth polytope satisfying one of the 
    following equivalent conditions:
    \begin{itemize}
        \item $|P \cap \Z^d| \leq 3d-4$,
        \item $\Vol(P) \leq 2d-4$.
    \end{itemize}
    Then $P$ has a quadratic triangulation. In particular it is IDP.
\end{cor}

Another consequence is the following finiteness result, independent
of the dimension.

\begin{cor}
    There are finitely many smooth polytopes of normalized volume $V$, for 
    any fixed integer $V>1$.
\end{cor}

By putting together this, together with our classification, we can easily classify
all smooth polytopes having normalized volume up to 10.

\begin{prop}
	\label{prop:small_smooth}
	Let $P$ be a smooth polytope having normalized volume at most 10.
	Then $P$ is either a Lawrence prism, or one of the following 14 polytopes:
	\begin{enumerate}
		\item[(2.a)] $\conv(\bzero,2\be_1,2\be_2)$,
		\item[(2.b)] $\conv(\bzero,3\be_1,3\be_2)$,
		\item[(2.c)] $\conv(\bzero,2\be_1,2\be_2,2\be_1+2\be_2)$,
		\item[(2.d)] $\conv(\bzero,\be_1,\be_2,-2\be_1 + \be_2, -4\be_1 + \be_2)$,
		\item[(2.e)] $\conv(\bzero,\be_1,\be_2,3\be_1+2\be_2,2\be_1+3\be_2,3\be_1+3\be_2)$,
		\item[(2.f)] $\conv(\bzero,\be_1,\be_2,2\be_1+\be_2,\be_1+2\be_2,2\be_1+2\be_2)$,
		\item[(2.g)] $\conv(\bzero,\be_1,\be_2,3\be_1+\be_2,\be_1+2\be_2,4\be_1+2\be_2)$,
		\item[(2.h)] $\conv(\bzero,\be_1,\be_1+2\be_2,-2\be_1+2\be_2)$,
		\item[(3.a)] $\conv(\bzero,3\be_1,3\be_2,3\be_3)$,
		\item[(3.b)] $\conv(\bzero,\be_1,\be_2,\be_1+\be_2,\be_3,\be_1+\be_3,-2\be_2+\be_3,\be_1-2\be_2+\be_3)$,
		\item[(3.c)] $\conv(\bzero,\be_1,\be_2,\be_3,-2\be_2+\be_3,2\be_1-2\be_2+\be_3)$,
		\item[(3.d)] $\conv(\bzero,\be_1,\be_2,\be_1+\be_2,\be_3,\be_1+\be_3,\be_2+\be_3,\be_1+\be_2+\be_3)$,
		\item[(4.a)] $\conv(\bzero,\be_1,\be_2,\be_3,\be_4,\be_1+\be_3,\be_1+\be_4,\be_2+\be_3,\be_2+\be_4)$,
		\item[(5.a)] $\conv(\bzero,\be_1,\be_2,\be_3,\be_4,\be_5,\be_1+\be_4,\be_1+\be_5,\be_2+\be_4,\be_2+\be_5,\be_3+\be_4,\be_3+\be_5)$,
	\end{enumerate}
\end{prop}

\section{Classifications of 3-simplices with few interior lattice points}
\label{sec:3-simplices}

In this section we show how the classification of polytopes (in particular 
simplices) with small volume can be used to enumerate all those having 
a fixed small number of interior lattice points. 

It is natural to wonder how large can the volume of a $d$-dimensional 
polytope $P$ be, when we fix the number of its interior lattice points to be
a nonnegative integer $k$. 
In case $k=0$, i.e. when $P$ is hollow, the answer is clear.
One can indeed fit an arbitrarily large $d$-dimensional hollow polytope in 
the ``slab'' $[0,1] \times \R^{d-1}$ . 
But in case $k \neq 0$ the answer is different:
Hensley~\cite{Hen83} proved that if $P$ is not hollow, then 
its volume is bounded by a constant depending only on the dimension $d$ and
the number of interior lattice points $k$.
In dimension two this was already know, thanks to a sharp bound proven by 
Scott in 1976 (see Theorem~\ref{thm:scott}, in the following section).
Hensley's bound has been improved first by Lagarias--Ziegler~\cite{LZ91}, 
later by Pikhurko~\cite{Pik01}, who shown the best currently known bound.
 
\begin{thm}
    \label{thm:pik}
    Let $P$ be a $d$-dimensional lattice polytope having $k$ interior lattice 
    points, $k\geq 1$. Then:
    \[
        \Vol(P) \leq d! \cdot(8d)^d \cdot 15^{d\cdot 2^{2d+1}}\cdot k
    \]
\end{thm}

Although it grows linearly with $k$ (which is the conjectured behaviour, as 
stated later in Conjecture~\ref{conj:volume_bound}), the bound is expected to 
be very rough.
The current largest known volume of a $d$-dimensional lattice polytope having 
$k$ interior points, $k\geq 1$, is given by the \emph{ZPW simplex} defined as
\begin{equation}
    S^d_k:=\conv(\bzero,s_1 \be_1,\ldots,s_{d-1}\be_{d-1},(k+1)(s_d-1)\be_d)
\end{equation}
first described by Zaks--Perles--Wills~\cite{ZPW82} (hence the name\footnote{
    In a personal communication, J. M. Wills explained that the 
    unusual order of the authors of the two pages paper was agreed by the three
    in order to allow J. Zaks to be the first author of a coauthored paper, at 
    least once. 
}).
Here $(s_i)_{i\in\Z_{\geq 1}}$ is the \emph{Sylvester sequence}, 
defined by the following recursion
\[
    s_1=2,\qquad s_i=s_1\cdots s_{i-1}+1.
\]
It has been conjectured that, once $d$ and $k$ are fixed, the ZPW simplex 
$S_k^d$ maximises the volume amongst all $k$-point $d$-dimensional polytopes. 
This conjecture has been explicitly stated in \cite{BK16}, but has been 
already hinted in some of the previously cited works~\cite{ZPW82,Hen83,LZ91}.

\begin{conj}[{\cite[Conjecture~1.5]{BK16}}]
    \label{conj:volume_bound}
    Fix $d\geq 3$ and $k\geq 1$. A $d$-dimensional lattice polytope $P$ having 
    $k$ interior lattice points satisfies
    \begin{equation}\label{eq:conj}
        \Vol(P)\leq (k+1)(s_d-1)^2.
    \end{equation}
    With the exception of the case when $d=3$, $k=1$, this inequality is an 
    equality if and only if $P=S^d_k$.
\end{conj}

The case when $d=3$, $k=1$ has been addressed in~\cite{Kas10}: in addition 
to the ZPW simplex $S^3_1$, the maximum normalized volume of $72$ is also 
attained by the simplex
\[
    \conv((0,0,0),(2,0,0),(0,6,0),(0,0,6)).
\]

In recent years, Conjecture~\ref{conj:volume_bound} has been proven for 
several families of lattice polytopes.
Explicit classifications \cite{Kas10,BK16} settle the cases $d=3$ and 
$k \in \{1,2\}$.
Averkov--Kr\"{u}mpelmann--Nill~\cite{AKN15} proved it for simplices with
one interior point, while Balletti--Kasprzyk--Nill \cite{BKN16} for
reflexive polytopes.
Recently Averkov \cite{Ave18} proved it for simplices having a facet with one 
lattice point in its relative interior.
In all these families the bound for the volume is sharp as they include 
the ZPW simplices.

We now use our classification to enumerate all the three-dimensional simplices
having few interior lattice points, in this way we will be able to verify
Conjecture~\ref{conj:volume_bound} in this additional cases.
The idea is to use volume bounds to make sure that our classification contains
all lattice polytopes having small numbers of interior points.
Note that using the general bounds, even the best ones known
(Theorem~\ref{thm:pik}) would be futile: one would have to classify all the
simplices having up to normalized volume $3.4 \cdot 10^{456}$, in order to be
sure that the classification contains all the three-dimensional polytopes 
with one interior lattice point.
Luckily some better bounds are known in special cases.

\begin{thm}[{\cite{Pik01}}]
    Let $S$ be a three-dimensional lattice simplex having $k$-interior lattice 
    points, with $k \geq 1$. Then
    \[
        \Vol(S) \leq \frac{29791}{352} k \leq 85k.
    \]
\end{thm}

We now use Algorithm~\ref{alg:sim} to classify all the elements in
$\Sim_{\leq 1000}^3$.

\begin{prop}
    There are $28 \, 015 \, 923$ three-dimensional simplices having normalized 
    volume at most $1 \, 000$.
\end{prop}

As an immediate corollary, we are able to fully enumerate the three-dimensional
simplices having up to 11 interior lattice points.

\begin{cor}
    \label{cor:simplices}
    All the 3-simplices having up to 11 interior lattice points are in 
    $\Sim_{\leq 1000}^3$.
    Their distribution by number of interior lattice points is summarized in
    Table~\ref{tab:int_pts}.
\end{cor}

Note that Corollary~\ref{cor:simplices} can be seen as an extension of existing classifications
of three-dimensional simplices performed up to two interior lattice points. The
225 three-dimensional simplices with one interior lattice points have been
enumerated by Borisov and Borisov \cite[pg.~278]{BB92}, while the 471 with
two interior lattice points are classified in \cite[Theorem~3.4]{BK16}. 

\begin{center}
    \small
    \begin{longtable}{c@{\ \ }c@{\ \ }c}
    \caption{
        Number of three-dimensional polytopes and simplices having few
        interior lattice points.
        The numbers of the polytopes are from \cite{Kas10} and \cite{BK16}.
    }
    \label{tab:int_pts}\\
    \toprule
    \begin{tabular}{@{}c@{}} number of \\ interior points \end{tabular} & 
    \begin{tabular}{@{}c@{}} number of \\ simplices       \end{tabular} & 
    \begin{tabular}{@{}c@{}} number of \\ polytopes       \end{tabular} \\
    \cmidrule(lr){1-3}
    \endfirsthead
    \orow 1  & 225  & 674688   \\
    \erow 2  & 471  & 22673449 \\
    \orow 3	 & 741  &		   \\
    \erow 4	 & 1206 &		   \\
    \orow 5	 & 1338 &		   \\
    \erow 6	 & 2063 &		   \\
    \orow 7	 & 2191 &		   \\
    \erow 8	 & 3007 &		   \\
    \orow 9	 & 3257 &		   \\
    \erow 10 & 4216 &		   \\
    \orow 11 & 4087 &		   \\
    \bottomrule
    \end{longtable}
\end{center}

Conjecture~\ref{conj:volume_bound} can be now verified on the classified 
objects. 
In a similar way, we are going to verify that the classify simplices 
satisfy a conjecture known as Duong Conjecture \cite[Conjectures~1-2]{Duo08}
It concerns a special class of three-dimensional lattice simplices.
We call a lattice polytope \emph{clean}, if the only lattice points on its 
boundary are the vertices.

\begin{conj}[{Duong Conjecture}]
    \label{conj:duong}
    Let $S$ be a three-dimensional clean simplex having $k$ interior lattice 
    points, with $k \geq 1$.
    Then 
    \[
        \Vol(S) \leq 12k + 8,
    \]
    where the equality is attained if and only if $S$ is the \emph{Duong 
    simplex}, defined as
    \[
        D^3_k\coloneq\conv\left((0,0,0),(1,0,0),(0,1,0),(3,6k+1,12k+8)\right).
    \]
\end{conj}

\begin{thm}
    Conjectures~\ref{conj:volume_bound} and~\ref{conj:duong} hold for 
    three-dimensional simplices having up to 11 interior lattice points.
\end{thm}

\section{ Conjectural Ehrhart inequalities in dimension three}
\label{sec:conj}

In this section we use the classification of three-dimensional polytopes to 
estimate the behaviour of their $h^*$-polynomials. 
We begin with a quick introduction to Ehrhart Theory, but we refer the 
interested reader to Beck and Robins' book \cite{BR15}.

Given a $d$-dimensional lattice polytope $P$ in $\R^d$, one can associate 
a function $t \mapsto |tP \cap \Z^d|$, which counts the number of lattice 
points in $tP$, the $t$-th dilation of $P$, where $t$ is a positive integer. 
Ehrhart~\cite{Ehr62} proved that this function behaves polynomially, i.e. 
that there exists a polynomial $\ehr_P$, that we call \emph{Ehrhart polynomial}
of $P$, satisfying $\ehr_P(t)=|tP \cap \Z^d|$ for $t \geq 1$.

Its generating function is the rational function
\[
    \sum_{t \geq 0} \ehr_P(t) z^t =\frac{\sum_{i \geq 0}\h_i z^i}{(1-z)^{d+1}},
\]
where $\h_i=0$ for any $i \geq d+1$.
We call the polynomial $\h_P(z)=\h_0 + \h_1 z + \cdots + \h_s z^s$ the 
\emph{$\h$-polynomial} of $P$ (sometimes also called 
\emph{$\delta$-polynomial}), and we set the \emph{degree} of $P$ to be $s$,
the degree of its $\h$-polynomial. 
In the following we will often identify the $\h$-polynomial with the vector of 
its coefficients $(\h_0,\h_1,\ldots,\h_d)$, which is called the 
\emph{$\h$-vector} (or \emph{$\delta$-vector}) of $P$. Some properties of the 
$\h$-polynomial are well known, and listed in the following proposition.

\begin{prop}
    \label{prop:h^*_basic}
    The coefficients $h^*_0,h^*_1,\ldots,h^*_d$ of the $h^*$-polynomial of $P$ 
    are nonnegative integers satisfying the following conditions:
    \begin{enumerate}
        \item $\h_0=1$;
        \item $\h_1=|P \cap \Z^d| - d - 1$;
        \item $\h_d = |P^\circ \cap \Z^d|$;
        \item $\sum_{i=0}^d \h_i=\Vol(P)$.
    \end{enumerate}
\end{prop}

The nonnegativity part of the statement above is a result of Stanley 
\cite{Sta80}, while the rest can be derived from Ehrhart's original approach. 
Combinatorial interpretations for the other coefficients are possible, but they
are not as natural as the ones above. 
One of the biggest challenges in Ehrhart Theory is to characterize the $
\h$-vectors of lattice polytopes.

\begin{question}
    \label{question}
    For each $d$, which vectors $(\h_0,\h_1,\ldots,\h_d)$ are $\h$-vectors 
    of some $d$-dimensional lattice polytope? 
\end{question}

This question is broadly open even in dimension three.
In two dimensions the answer to Question~\ref{question} was first given by
Scott~\cite{Sco76}.

\begin{thm}
    \label{thm:scott}
    The vector with integer entries $(1,\h_1,\h_2)$ is the $\h$-vector of a 
    two-dimensional lattice polytope if and only if one of the following 
    conditions holds:
    \begin{center}
        \begin{enumerate}
            \item $\h_2=0$;
            \item $0 < \h_2\leq\h_1\leq 3\h_2+3$;
            \item $(1,\h_1,\h_2)=(1,7,1)$.
        \end{enumerate}
    \end{center}
    The last condition is attained only by the exceptional triangle 
    \[2\D_2 =\conv(\bzero,2\be_1,2\be_2).\]
\end{thm}

A generalized version of Theorem~\ref{thm:scott} has been proven in each
dimension and for degree two polytopes in \cite{Tre10}. This has been
generalized further to each degree in \cite{BH17}, showing that there are 
inequalities for the $\h$-coefficients that are \emph{universal}, i.e. not 
depending on dimension and degree.

Some relation among the coefficients of $\h$-polynomial are known, and
summarized in the following proposition. The first one can be the deduced
directly from Proposition~\ref{prop:h^*_basic}, the others come from works
by Stanley~\cite{Sta91} and Hibi~\cite{Hib94}.

\begin{prop}
    \label{prop:h^*-ineq}
    Let $P$ be a $d$-dimensional lattice polytope of degree $s$ with its 
    $h^*$-vector $h^*(P)=(h^*_0,\ldots,h^*_d)$.
    Then:
    \begin{enumerate}
        \item $\h_d \leq \h_1$;
        \item $\h_{d-1} + \cdots + \h_{d-i} \leq \h_2 + \dots + \h_{i+1}$, for 
              $i=1,\ldots, d-1$;
        \item $\h_{0} + \dots + h_{i} \leq h_{s} + \cdots + h_{s-i}$, for 
              $i=0,1,\ldots, s$;
        \item if $\h_d > 0$, then $\h_1 \leq \h_i$, for $i=1,2,\ldots,d-1$.
    \end{enumerate}
\end{prop}

In recent years Stapledon~\cite{Sta09,Sta16} showed the existence of infinite 
new classes of inequalities, giving explicit formulae in small dimensions.

In dimension three, the known inequalities are far from giving a complete
picture of the possible $\h$-vectors.
Nevertheless this case has been solved in the case $\h_3=0$, i.e. for hollow
lattice polytopes. 
In this case a polytope has degree two, and Treutlein's result \cite{Tre10} 
gives a necessary condition, while Henk--Tagami \cite{HT09} prove the 
sufficiency part.

\begin{thm}[{\cite[Theorem~2]{Tre10},\cite[Proposition~2.10]{HT09}}]
    The vector with integer entries $(1,\h_1,\h_2,0)$ is the $\h$-vector
    of a three-dimensional lattice polytope if and only if one of the 
    following conditions holds:
    \begin{enumerate}
        \item $\h_2 = 0$;
        \item $ 0 \leq \h_1 \leq 3\h_2 + 3$;
        \item $(1,\h_1,\h_2,0) = (1,7,1,0)$. 
    \end{enumerate}
    The last condition is attained only by the exceptional simplex
    \[\conv(\bzero,2\be_1,2\be_2,\be_3).\]
\end{thm}

In the following we use the classification of three-dimensional
lattice polytopes to conjecture a set of sharp inequalities describing
the behaviour of three-dimensional polytopes with interior points, i.e.
the polytopes whose Ehrhart coefficient $\h_3$ is nonzero.
The immediate way to do this is by plotting the $\h$-vectors of classified
polytopes (see Appendix~\ref{app:plots}), and try to understand
which inequalities seem to be satisfied.

In what follows, we frequently need to calculate the $\h$-vector of families 
of lattice simplices depending on a parameter.
The following lemma is an example.
Its proof outlines how these kind of results can be proven, and similar
results will be given without proof in the rest of the section.

\begin{lem}
	\label{lem:example}
	The ZPW simplex  
	    \[
	        S^3_k \coloneq \conv\left((0,0,0),(2,0,0),(0,3,0),(0,0,6k+6)\right),
	    \]
	    has $\h$-vector $(1,16k+19,19k+16,k)$. 
\end{lem}
\begin{proof}
    From Proposition~\ref{prop:h^*_basic}, we can write the $\h$-vector of 
    any three-dimensional lattice polytope $P$ in terms of number of lattice 
    points, number of interior lattice points and volume as
    \[
        (1,|P \cap \Z^d|-4, \Vol(P)-|P \cap \Z^d|-|P^\circ \cap \Z^3|+3,
        |P^\circ \cap \Z^3).
    \]
    The normalized volume $\Vol(S^3_k)$ of $S^3_k$ can be calculated trivially 
    and equals $36(k+1)$.
    For the number of interior lattice point we project $S^3_k$ along $\be_3$,
    and we deduce that they have to be all of the form $(1,1,a)$ for $a \geq 1$.
    Note that $(1,1,k+1)$ can be written as a convex combination of the 
    vertices $(2,0,0)$,$(0,3,0)$ and $(0,0,6k+6)$ with weights $\frac{1}{2}$,
    $\frac{1}{3}$,$\frac{1}{6}$, respectively, in particular its is on the 
    boundary of $S^3_k$.
    Therefore the interior points are all those of the form $(1,1,a)$, with
    $1 \leq a \leq k$, and they are exactly $k$.
    With a similar (and easier) argument, one can count the lattice points in
    the relative interior of lower dimensional faces of $S^3_k$.
    By summing everything up, we get $|P \cap \Z^d|=16k+23$, and hence the 
    thesis. 
\end{proof}

We conjecture the following.

\begin{conj}
    \label{conj:main}
    Let $P$ be a three-dimensional lattice polytope having at least one
    interior lattice point.
    Then its $\h$-vector $(1,\h_1,\h_2,\h_3)$ satisfies the following 
    inequalities.
    \begin{enumerate}
        \item $\h_3 \leq \h_1$,
        \item $\h_1 \leq \h_2$,
        \item $\h_2 \leq 19\h_3+16$,
        \item \label{ineq:special} $\h_2 - \h_1 \leq 9 \h_3 + 9$,
        \item $ 5 \h_3 \h_1  + 4 \h_1 + 4 \leq 
              4 {\h_3}^2 + 4 \h_3 \h_2 + 5 \h_2$.
    \end{enumerate}
    Moreover, the fourth inequality holds in the stronger form
    \footnote{
        This inequality was also conjectured by M\'onica Blanco and Lukas 
        Katth\"an (private communication). 
        They expressed it in an equivalent Pick-like form 
        $\Vol(P) \leq 2b + 12 i$ where $b$ and $i$ are the 
        number of lattice points of $P$ on its boundary and in its interior 
        respectively.
    }
    \begin{enumerate}
        \item[(4*)] $\h_2 - \h_1 \leq 9 \h_3 + 7$, 
    \end{enumerate}
    unless $P$ is one of the following exceptional cases (listed together with
    their $\h$-vectors):
    \begin{enumerate}[(i)]
    \item $\conv\left((0,0,0),(1,0,0),(0,1,0),(2,19,25)\right)$, 
          \hfill $(1,3,20,1)$
    \item $\conv\left((0,0,0),(1,0,0),(0,1,0),(3,19,28)\right)$, 
          \hfill $(1,4,22,1)$
    \item $\conv\left((0,0,0),(1,0,0),(0,1,0),(3,13,19),(1,-2,-3)\right)$, 
          \hfill $(1,5,22,1)$
    \item $\conv\left((0,0,0),(1,0,0),(0,1,0),(4,7,11),(-5,-7,-15)\right)$, 
          \hfill $(1,7,24,1)$
    \item $\conv\left((0,0,0),(1,0,0),(0,1,0),(2,2,3),(-21,-8,-25)\right)$,
          \hfill $(1,11,28,1)$
    \item $\conv\left((0,0,0),(1,0,0),(0,1,0),(5,17,42)\right)$, 
          \hfill $(1,11,29,1)$
    \item $\conv\left((0,0,0),(1,0,0),(0,1,0),(5,7,17),(1,-2,-5)\right)$, 
          \hfill $(1,12,29,1)$
    \item $\conv\left((0,0,0),(1,0,0),(0,1,0),(7,2,9),(-7,-3,-15)\right)$, 
          \hfill $(1,13,30,1)$
    \item $\conv\left((0,0,0),(1,0,0),(0,1,0),(0,0,1),(5,42,-25)\right)$, 
          \hfill $(1,14,31,1)$
    \item $\conv\left((0,0,0),(1,0,0),(0,1,0),(5,23,45)\right)$, 
          \hfill $(1,8,34,2)$
    \item $\conv\left((0,0,0),(1,0,0),(0,1,0),(3k+3,9k+8,18k+15)\right)$,
          \hfill $(1,4k+3,13k+11,k)$
    \item $\conv\left((0,0,0),(1,0,0),(0,1,0),(3,12k+8,18k+15)\right)$,
          \hfill $(1,4k+3,13k+11,k)$
    \end{enumerate}
    
    where, in the last two cases, $k \in \Z_{\geq 1}$.
    Inequalities (3) and (5) are both attained with equality if 
    and only if, for some $k \geq 1$, $P$ is one of the ZPW simplices
    \[
        S^3_k \coloneq \conv\left((0,0,0),(2,0,0),(0,3,0),(0,0,6k+6)\right),
    \]
    having $\h$-vector $(1,16k+19,19k+16,k)$, or the special ``almost-ZPW'' 
    simplex
    \[
        \tilde{S}^3_1\coloneq\conv\left((0,0,0),(2,0,0),(0,6,0),(0,0,6)\right),
    \]
    having $\h$-vector $(1,35,35,1)$.

    Inequalities (1) and (4*) are both attained with equality if and
    only if, for some $k\geq 1$, $P$ is one of the Duong simplices
    \[
        D^3_k\coloneq\conv\left((0,0,0),(1,0,0),(0,1,0),(3,6k+1,12k+8)\right)
    \]
    having $\h$-vector $(1,k,10k+7,k)$.
\end{conj}

Inequalities \textit{(1)} and \textit{(2)} of Conjecture~\ref{conj:main} are 
already known to be true, as they are a consequence of 
Proposition~\ref{prop:h^*-ineq}. 
Note that Conjecture~\ref{conj:main} generalizes the three-dimensional case of 
Conjecture~\ref{conj:volume_bound}, for the maximal volume of polytopes 
having interior lattice points, and Conjecture~\ref{conj:duong} (Duong 
Conjecture).
Moreover it generalizes Conjecture~6.1 of \cite{BK16}, on the maximal 
$\h$-coefficients of lattice polytopes with interior lattice points.

Note that inequality \textit{(5)} of Conjecture~\ref{conj:main} is non-linear.
However, fixing $\h_3$ to be larger than one, the inequalities are linear and 
define a pentagon (see Figure~\ref{fig:conj}).
In the special case when $\h_3=1$ equalities \textit{(2)} and \textit{(5)} 
coincide.
We now give a vertex representation of such a pentagon.

\begin{prop}
    \label{prop:vs}
    For each $\h_3 \in \Z_{>1}$, inequalities (1),(2),(3),(4*),(5) define
    a pentagon in $\R^4$ with vertices
    \begin{align*}
        \bv{1} & \coloneq (1,\h_3,\h_3,\h_3),\\
        \bv{2} & \coloneq (1,4\h_3+4,4\h_3+4,\h_3),\\
        \bv{3} & \coloneq (1,16\h_3+19,19\h_3+16,\h_3),\\
        \bv{4} & \coloneq (1,10\h_3+9,19\h_3+16,\h_3),\\
        \bv{5} & \coloneq (1,\h_3,10\h_3+7,\h_3).
    \end{align*}
    Moreover,
    \begin{itemize}
        \item $\bv_1$ is realized as the $\h$-vector of the polytope
              \[
                  \conv\left((0,0,0),(1,0,0),(0,1,0),(3,3\h_3,3\h_3+1)\right),
              \]
        \item $\bv_2$ is realized as the $\h$-vector of the polytope
              \[
                  \conv\left((0,0,0),(1,0,0),(2,3,0),(2,3,3+3\h_3)\right),
              \]
        \item $\bv_3$ is realized as the $\h$-vector of the ZPW simplex 
              $S^3_{\h_3}$,
        \item $\bv_5$ realized as the $\h$-vector of the Duong simplex 
              $D^3_{\h_3}$.
          \end{itemize}
\end{prop}

\begin{figure}[ht]
    \centering
    \includegraphics{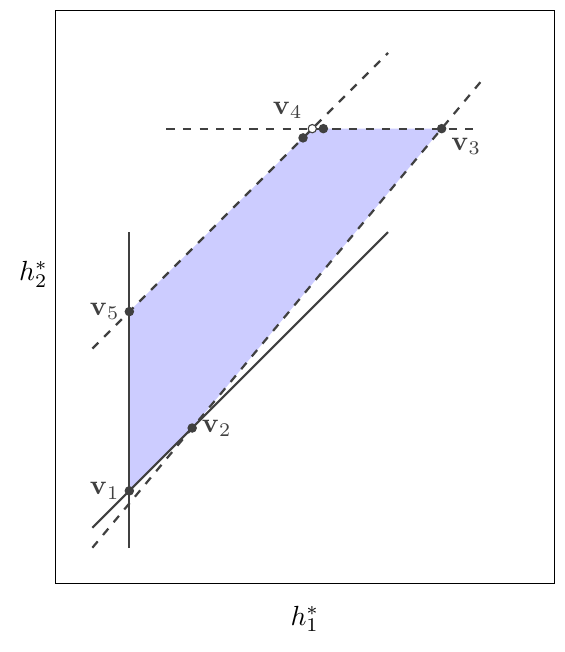}
    \caption{
        The pentagon defined by Conjecture~\ref{conj:main}, for a fixed $\h_3$.
        The dashed lines represent inequalities \textit{(3)}, \textit{(4*)} and 
        \textit{(5)}, which are conjectural.
        Note that the vertex $\bv_4$ seems not to be realized as the $\h$-vector
        of any lattice polytope, anyway there are other realizable $\h$-vector 
        which are very close to it (see Proposition~\ref{prop:v4} and the 
        previous discussion).
    }
    \label{fig:conj}
\end{figure}

Note that we do not have a candidate polytope realizing the vertex $\bv_4$, for
each value of $\h_3$. 
From the data available, it seems like such $\h$-vector is never attained for 
any number of interior point. 
Anyway, in order to show that inequalities $\textit{(3)}$ and \textit{(4*)} are
actually sharp, we give examples of polytopes having, for each $\h_3$, 
$\h$-vector ``close enough'' to $\bv_4$.
We remark that, in this way, there should be an additional inequality in 
Conjecture~\ref{conj:main} cutting out the vertex $\bv_4$, and creating two
additional ones.
Anyway the distance of the two new vertices from $\bv_4$ is fixed and does not
depend on the value of $\h_3$, so we decided not to include it.

\begin{prop}
    \label{prop:v4}
    For each positive integer value of $\h_3$, the $\h$-vectors
    \begin{itemize}
        \item $(1,10\h_3+11,19\h_3+16,\h_3)$,
        \item $(1,10\h_3+9,19\h_3+15,\h_3)$,
        \item $(1,10\h_3+7,19\h_3+14,\h_3)$,
    \end{itemize}
    are attained, respectively, by the polytopes
    \begin{itemize}
        \item
            $
            \!
            \begin{aligned}[t]
                \conv\left((1,0,0),(2,0,0),(0,1,0),(0,3,0),(0,0,6\h_3+5),
                (1,0,3\h_3+3)\right),
            \end{aligned}    
            $
        \item
            $
            \!
            \begin{aligned}[t]
                \conv(&(0,0,0),(1,0,0),(0,1,0),(9\h_3+8,6\h_3+5,18\h_3+15),\\
                      &(12\h_3+10,8\h_3+7,24\h_3+20)),
            \end{aligned}    
            $
        \item
            $
            \!
            \begin{aligned}[t]
                \conv(&(0,0,0),(1,0,0),(0,1,0),(6\h_3+5,3\h_3+3,18\h_3+15),\\
                      &(8\h_3+5,4\h_3+3,24\h_3+14)).
            \end{aligned}    
            $
    \end{itemize}
\end{prop}
\begin{proof}
    Using the same technique used in Lemma~\ref{lem:example}, one can prove that the simplices
    \begin{itemize}
        \item $S^3_{\h_3}$,  
        \item $\conv\left(
              (0,0,0),(1,0,0),(0,1,0),(12\h_3+10,6\h_3+6,36\h_3+30)
              \right)$,  
        \item $\conv\left(
              (0,0,0),(1,0,0),(0,1,0),(12\h_3+8,6\h_3+5,36\h_3+24)
              \right)$,  
          \end{itemize}
    have $\h$-vectors respectively equal to
    \begin{itemize}
        \item $(1,16\h_3+19,19\h_3+16,\h_3)$,
        \item $(1,16\h_3+14,19\h_3+15,\h_3)$,
        \item $(1,16\h_3+ 9,19\h_3+14,\h_3)$.
    \end{itemize}
    From this, one can obtain the three polytopes starting from the simplices,
    and progressively cutting out unimodular simplices. 
    At each cut, $\h_1$ drops by one, while $\h_2$ stays the same.
\end{proof}

As a final observation for this section, we plot heat diagrams of the
distribution of $\h$-vectors of three-dimensional lattice polytopes having one 
and two interior lattice points. 
From Figure~\ref{fig:1-2} one can note that Conjecture~\ref{conj:main} seems 
to describe accurately the behaviour of $\h$-vectors in dimension three, and 
that most of the $\h$-vectors seems to be in the center of the pentagon.

\begin{figure*}[ht]
    \centering
    \begin{adjustbox}{minipage=\linewidth,scale=0.9}
    \begin{subfigure}[t]{0.45\textwidth}
    \centering
    \includegraphics{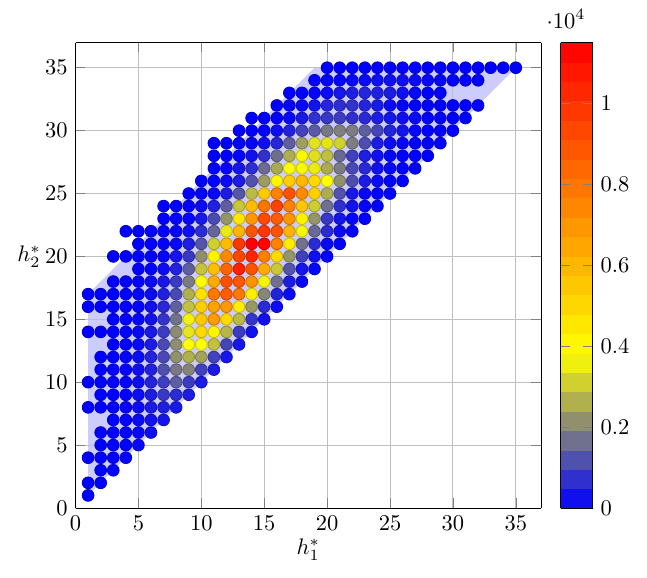}
    \caption{$\h_3 = 1$}
    \end{subfigure}%
    ~
    \begin{subfigure}[t]{0.45\textwidth}
    \centering
    \includegraphics{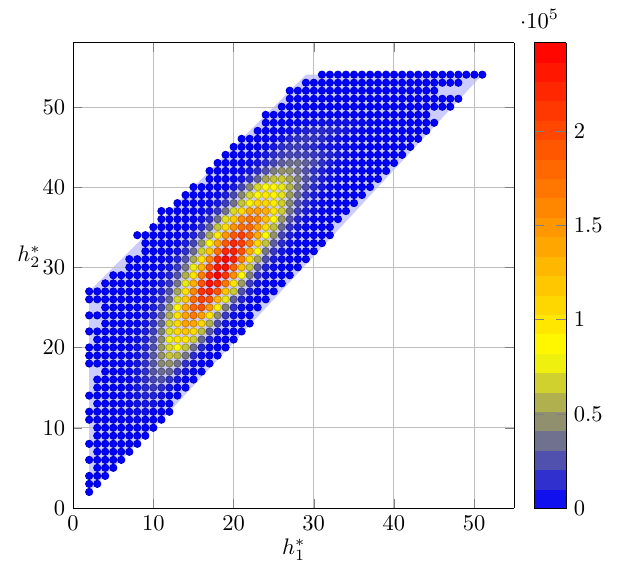}
    \caption{$\h_3 = 2$}
    \end{subfigure}
    \end{adjustbox}
    \caption{
        Heat diagrams for $\h$-vector of three-dimensional polytopes with
        one and two interior lattice points (from \cite{Kas10,BK16}).
        Each $\h$-vector is coloured according to the number of lattice 
        polytopes attaining it.
    }
    \label{fig:1-2}
\end{figure*}

\section{Final examples}
\label{sec:examples}

In this final section we use the classification to explicitly study how common
are some of the most studied properties of lattice polytopes.

In the literature it is possible to find a vast multitude of hierarchies
ordering lattice polytopes with a chain of progressively more restricting
properties.
Having the Integer Decomposition Property (defined in Section~\ref{sec:smooth}),
plays usually a central role in such a hierarchy, due to its
importance in algebraic and optimization contexts.
Here we additionally focus on the following properties

\begin{defn}
	Let $P \subset \R^d$ be a $d$-dimensional lattice polytope. We say that
	$P$ is \emph{spanning} if its lattice points affinely span $\Z^d$, i.e. if
    \[
	    \Z((P-\bp) \cap \Z^d) = \Z^d,
    \]
    for any $\bp \in P \cap \Z^d$.
    We say that $P$ \emph{very ample} if for each vertex $\bv$ of $P$ the
    lattice points in the tangent cone $\R_{\geq 0} (P-\bv)$ are sums of
    lattice points of $P-\bv$, i.e if
    \[
	    \R_{\geq 0} (P-\bv) \cap \Z^d = \Z_{\geq 0}((P-\bv) \cap \Z^d).
    \]
    We say that $P$ \emph{has a unimodular cover} if there exist
    unimodular lattice simplices $S_1,\ldots,S_n \subset \R^d$ such that
    $P =  S_1 \cup \cdots \cup S_n$.
    Finally, we say that $P$ \emph{has a unimodular triangulation} if $P$
    admits a triangulation in unimodular lattice simplices.
\end{defn}
It is easy to verify that such properties are given in ascending order
of restrictiveness, with the IDP property being in the middle, i.e

\vspace{0.2cm}
\begin{center}
	\begin{tabular} {@{}c@{}}unimodular \\ triangulation \end{tabular}
	$\Rightarrow$
	\begin{tabular} {@{}c@{}}unimodular \\ cover \end{tabular}
	$\Rightarrow$
	IDP
	$\Rightarrow$
	very ample
	$\Rightarrow$
	spanning.
\end{center}
\vspace{0.2cm}
Notice that in dimension two all these properties are always
satisfied by all polytopes.
On the other hand, in higher dimensions the opposite inequalities
are known to be false.
While counterexamples to the last implication (very ampleness implies
spanning) are easy to produce, 
the first example of a very ample but not IDP lattice polytope
was given in dimension five by Bruns and Gubeladze \cite{BG02}, who
later gave an example in dimension three \cite[Exercise~2.24]{BG09}.
Examples of IDP polytopes not having a unimodular cover have also
been given by Bruns and Gubeladze in dimension five \cite{BG99}.
The first example of a three-dimensional polytope which has a
unimodular but does not have a unimodular triangulation has been
given by Kantor and Sarkaria \cite{KS03}, although the first 
example, in dimension four, appears in \cite{BGT97}.

By looking at the database of three-dimensional polytopes classified
in this paper, we can easily find examples of lattice polytopes which are
spanning but not very ample, examples of lattice polytopes which are very
ample but not IDP, and examples IDP lattice polytopes not admitting a 
unimodular triangulation. 
Additionally we can be sure that the following examples are the smallest possible,
i.e. those having smallest possible dimension and volume.

\begin{thm}
	The polytope 
	\[
		\conv(\{\bzero,\be_1,\be_2,\be_3,\be_1+\be_2+3\be_3\})
	\]
	is spanning but not very ample. The polytope 
	\[
		\conv(\{\bzero,\be_1,\be_2,\be_1 + \be_2 + 2 \be_3,\be_1 + \be_2 + 3 \be_3,
		\be_1 - \be_3,\be_2 - \be_3,\be_3\})
	\]
	is very ample but not IDP. The polytope
	\[
		\conv(\{\bzero,\be_1,\be_2,\be_1 + \be_2, \be_3, \be_1 + 2 \be_2 - \be_3,
		3 \be_1 + \be_2 - \be_3, -2 \be_1 - \be_2 + \be_3\})
	\] 
	has a unimodular cover but does not have a unimodular triangulation.
	Such examples are those of minimal volume in dimension three.
\end{thm}

The last example also appears in \cite{DST01}.
No IDP polytope without a unimodular cover has been found in dimesion three,
but not all the enumerated three polytopes could be checked, as the algorithm implemented
for checking the existence of a unimodular cover is computationally expensive.
The presence of such polytopes in higher dimension has already been mentioned,
but it makes sense to wonder whether being IDP and having a unimodular cover are
equivalent properties in dimension three.

\begin{question}
Is there a three-dimensional IDP polytope that does not have a unimodular cover?
\end{question}

In this last part of the paper di discuss about properties of the $\h$-vectors 
of very ample lattice polytopes.
We call a sequence $a_0, a_1 , \ldots, a_n$ of real numbers \emph{unimodal}
if, for some $0 \leq k \leq n$,
\[
	a_0 \leq \cdots \leq a_{k-1} \leq a_k \geq a_{k+1} \geq \cdots \geq a_n.
\]
We call $a_0, a_1 , \ldots, a_n$ \emph{log-concave} if, for all $ 1 \leq i \leq n-1$
\[
	a_{i-1} a_{i+1} \leq a_i^2.
\]
If all the $a_i$ are nonnegative, then log-concavity implies unimodality.
It is a long standing open problem (originally posed by Stanley) to understand 
whether all IDP polytopes have unimodal, or even log-concave, $\h$-vector 
(see Braun's survey \cite{Br16}).
In \cite[p.39]{HHJN07} it is shown that, by relaxing the IDP property to very 
ampleness, it is possible to lose log-concavity.
This is done by giving an example of nine-dimensional very ample (but non IDP)
lattice polytope.

By looking at our database it turns out that this kind of examples can be small
and exist already in dimension three, as the polytope
\[
	\conv(\bzero,e_1,e_2,e_3,e_1+e_3,e_2+e_3,e_1+e_2+16e_3,e_1+e_2+17 e_3),
\]
is very ample and has $\h$-vector $(1,4,17,0)$, which is not log-concave.
This example can be easily generalized by considering the three-dimensional lattice
polytope
\[
	Q \coloneqq \conv(\bzero,e_1,e_2,e_3,e_1+e_3,e_2+e_3,e_1+e_2+ke_3,e_1+e_2+(k+1) e_3),
\]
where $k$ is a nonnegative integer.
It is easy to verify that $Q$ is a very ample lattice polytope with $\h$-vector
$(1,4,k+1,0)$, which, for $k \geq 16$, is not log-concave.
Note that $Q$ fails to be IDP for $k \geq 4$. This kind of constructions are
called \emph{segmental fibration} and have been used in \cite{BDGM15} to generate
non IDP but very ample polytopes.

\bibliographystyle{alpha}

\newcommand{\etalchar}[1]{$^{#1}$}

\newpage
\appendix
\section{Smooth polytopes}
\label{app:smooth}

\begin{center}
    \small
    \begin{longtable}{c@{\ \ }cc@{\ \ }cc@{\ \ }c}
    \caption{
        Distribution of two-dimensional smooth polytopes by their 
        normalized volume.
    }
    \label{tab:smooth-2p}\\
    \toprule
    volume & \begin{tabular}{@{}c@{}}smooth \\ polytopes\end{tabular} & 
    volume & \begin{tabular}{@{}c@{}}smooth \\ polytopes\end{tabular} & 
    volume & \begin{tabular}{@{}c@{}}smooth \\ polytopes\end{tabular} \\ 
    \cmidrule(lr){1-2} \cmidrule(lr){3-4} \cmidrule(lr){5-6}
    \endfirsthead
    \orow 1  & 1  & 18 & 15 & 35 & 42  \\
    \erow 2  & 1  & 19 & 16 & 36 & 41  \\
    \orow 3  & 1  & 20 & 18 & 37 & 35  \\
    \erow 4  & 3  & 21 & 13 & 38 & 60  \\
    \orow 5  & 2  & 22 & 23 & 39 & 53  \\
    \erow 6  & 4  & 23 & 21 & 40 & 56  \\
    \orow 7  & 4  & 24 & 24 & 41 & 41  \\
    \erow 8  & 6  & 25 & 19 & 42 & 63  \\
    \orow 9  & 5  & 26 & 26 & 43 & 61  \\
    \erow 10 & 7  & 27 & 25 & 44 & 62  \\
    \orow 11 & 7  & 28 & 30 & 45 & 61  \\
    \erow 12 & 9  & 29 & 22 & 46 & 91  \\
    \orow 13 & 7  & 30 & 39 & 47 & 66  \\
    \erow 14 & 12 & 31 & 34 & 48 & 72  \\
    \orow 15 & 12 & 32 & 34 & 49 & 78  \\
    \erow 16 & 15 & 33 & 27 & 50 & 111 \\
    \orow 17 & 9  & 34 & 46 &    &     \\
    \bottomrule
    \end{longtable}
\end{center}

\begin{center}
    \small
    \begin{longtable}{c@{\ \ }cc@{\ \ }cc@{\ \ }c}
    \caption{
        Distribution of three-dimensional smooth polytopes by their 
        normalized volume.
    }
    \label{tab:smooth-3p}\\
    \toprule
    volume & \begin{tabular}{@{}c@{}}smooth \\ polytopes\end{tabular} & 
    volume & \begin{tabular}{@{}c@{}}smooth \\ polytopes\end{tabular} & 
    volume & \begin{tabular}{@{}c@{}}smooth \\ polytopes\end{tabular} \\ 
    \cmidrule(lr){1-2} \cmidrule(lr){3-4} \cmidrule(lr){5-6}
    \endfirsthead
    \orow 1  &  1  & 13 & 16 & 25 &  56 \\
    \erow 2  &  0  & 14 & 17 & 26 &  63 \\
    \orow 3  &  1  & 15 & 22 & 27 &  79 \\
    \erow 4  &  1  & 16 & 22 & 28 &  72 \\
    \orow 5  &  2  & 17 & 25 & 29 &  74 \\
    \erow 6  &  4  & 18 & 36 & 30 & 103 \\
    \orow 7  &  5  & 19 & 33 & 31 &  89 \\
    \erow 8  &  6  & 20 & 35 & 32 &  92 \\
    \orow 9  &  8  & 21 & 47 & 33 & 115 \\
    \erow 10 &  8  & 22 & 43 & 34 & 109 \\
    \orow 11 &  10 & 23 & 48 & 35 & 113 \\
    \erow 12 &  16 & 24 & 66 & 36 & 151 \\
    \bottomrule
    \end{longtable}
\end{center}

\begin{center}
    \small
    \begin{longtable}{c@{\ \ }cc@{\ \ }cc@{\ \ }c}
    \caption{
        Distribution of four-dimensional smooth polytopes by their 
        normalized volume.
    }
    \label{tab:smooth-4p}\\
    \toprule
    volume & \begin{tabular}{@{}c@{}}smooth \\ polytopes\end{tabular} & 
    volume & \begin{tabular}{@{}c@{}}smooth \\ polytopes\end{tabular} & 
    volume & \begin{tabular}{@{}c@{}}smooth \\ polytopes\end{tabular} \\ 
    \cmidrule(lr){1-2} \cmidrule(lr){3-4} \cmidrule(lr){5-6}
    \endfirsthead
    \orow 1  &  1  & 9  & 6  & 17 & 40  \\
    \erow 2  &  0  & 10 & 9  & 18 & 49  \\
    \orow 3  &  0  & 11 & 12 & 19 & 54  \\
    \erow 4  &  1  & 12 & 16 & 20 & 66  \\
    \orow 5  &  1  & 13 & 18 & 21 & 73  \\
    \erow 6  &  3  & 14 & 23 & 22 & 86  \\
    \orow 7  &  3  & 15 & 28 & 23 & 94  \\
    \erow 8  &  5  & 16 & 36 & 24 & 114 \\
    \bottomrule
    \end{longtable}
\end{center}

\begin{center}
    \small
    \begin{longtable}{c@{\ \ }cc@{\ \ }cc@{\ \ }c}
    \caption{
        Distribution of five-dimensional smooth polytopes by their 
        normalized volume.
    }
    \label{tab:smooth-5p}\\
    \toprule
    volume & \begin{tabular}{@{}c@{}}smooth \\ polytopes\end{tabular} & 
    volume & \begin{tabular}{@{}c@{}}smooth \\ polytopes\end{tabular} & 
    volume & \begin{tabular}{@{}c@{}}smooth \\ polytopes\end{tabular} \\ 
    \cmidrule(lr){1-2} \cmidrule(lr){3-4} \cmidrule(lr){5-6}
    \endfirsthead
    \orow 1  & 1 & 8  & 3  & 15 & 30 \\
    \erow 2  & 0 & 9  & 5  & 16 & 38 \\
    \orow 3  & 0 & 10 & 8  & 17 & 47 \\
    \erow 4  & 0 & 11 & 10 & 18 & 57 \\
    \orow 5  & 1 & 12 & 13 & 19 & 70 \\
    \erow 6  & 1 & 13 & 18 & 20 & 85 \\
    \orow 7  & 2 & 14 & 23 &    &    \\
    \bottomrule
    \end{longtable}
\end{center}

\begin{center}
    \small
    \begin{longtable}{c@{\ \ }cc@{\ \ }cc@{\ \ }c}
    \caption{
        Distribution of six-dimensional smooth polytopes by their 
        normalized volume.
    }
    \label{tab:smooth-6p}\\
    \toprule
    volume & \begin{tabular}{@{}c@{}}smooth \\ polytopes\end{tabular} & 
    volume & \begin{tabular}{@{}c@{}}smooth \\ polytopes\end{tabular} & 
    volume & \begin{tabular}{@{}c@{}}smooth \\ polytopes\end{tabular} \\ 
    \cmidrule(lr){1-2} \cmidrule(lr){3-4} \cmidrule(lr){5-6}
    \endfirsthead
    \orow 1  & 1   & 7  & 1  & 13 & 14 \\
    \erow 2  & 0   & 8  & 2  & 14 & 20 \\
    \orow 3  & 0   & 9  & 3  & 15 & 27 \\
    \erow 4  & 0   & 10 & 5  & 16 & 35 \\
    \orow 5  & 0   & 11 & 7  &    &    \\
    \erow 6  & 1   & 12 & 11 &    &    \\
    \bottomrule
    \end{longtable}
\end{center}

\newpage
\section{Plots of $h^*$-vectors of three-dimensional simplices}
\label{app:plots}

\begin{figure*}[ht]
    \centering
    \caption{Plots of the distribution of the $\h$-vectors of three-dimensional
    simplices having up to eleven interior lattice points.
    The area defined by the inequalities of Conjecture~\ref{conj:main} is 
    marked in blue.}
    \begin{adjustbox}{minipage=\linewidth,scale=0.8}
    \begin{subfigure}[t]{0.3\textwidth}
    \centering
    \includegraphics{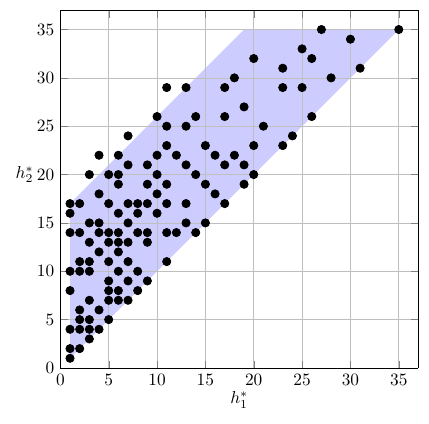}
    \caption{$\h_3 = 1$}
    \end{subfigure}
    ~
    \begin{subfigure}[t]{0.3\textwidth}
    \centering
    \includegraphics{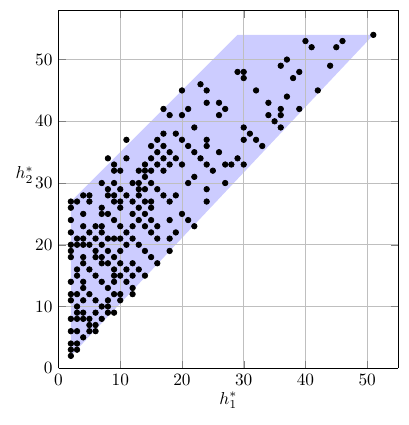}
    \caption{$\h_3 = 2$}
    \end{subfigure}
    ~
    \begin{subfigure}[t]{0.3\textwidth}
    \centering
    \includegraphics{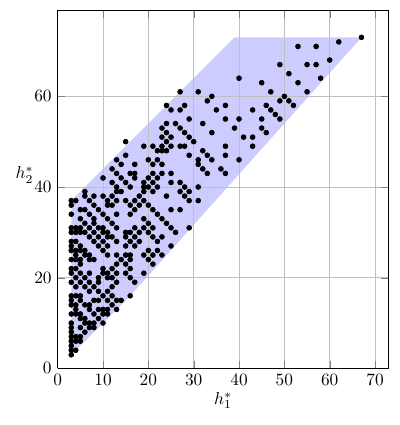}
    \caption{$\h_3 = 3$}
    \end{subfigure}
	\\
	\centering
    \begin{subfigure}[t]{0.3\textwidth}
    \centering
    \includegraphics{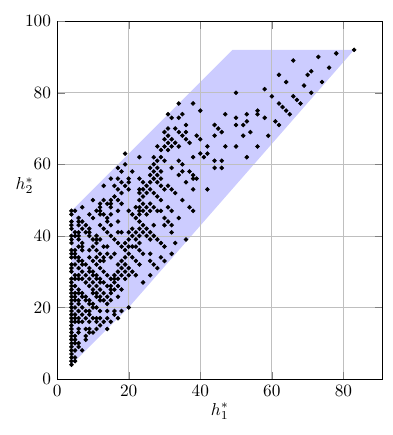}
    \caption{$\h_3 = 4$}
    \end{subfigure}
    ~
    \begin{subfigure}[t]{0.3\textwidth}
    \centering
    \includegraphics{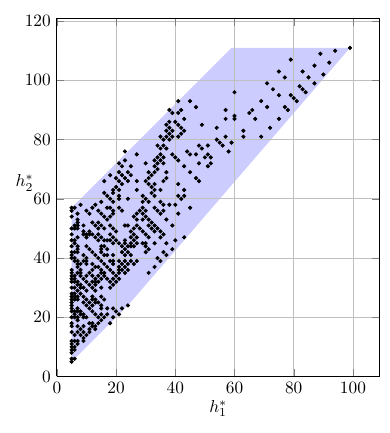}
    \caption{$\h_3 = 5$}
    \end{subfigure}
    ~
    \begin{subfigure}[t]{0.3\textwidth}
    \centering
    \includegraphics{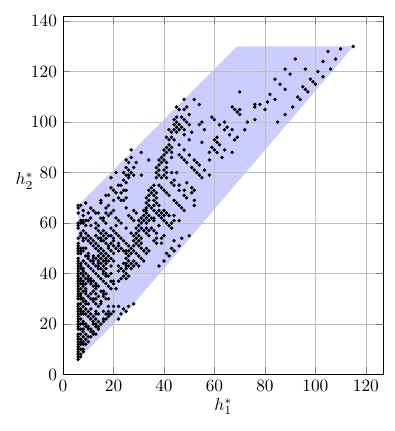}
    \caption{$\h_3 = 6$}
    \end{subfigure}
	\\
	\centering
    \begin{subfigure}[t]{0.3\textwidth}
    \centering
    \includegraphics{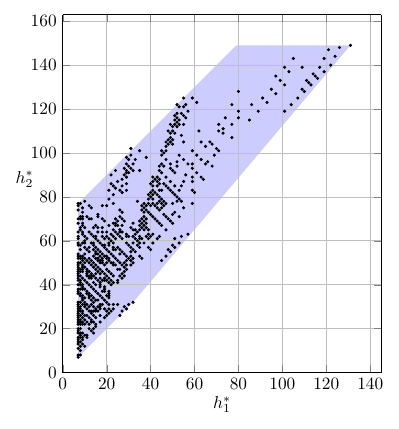}
    \caption{$\h_3 = 7$}
    \end{subfigure}%
    ~
    \begin{subfigure}[t]{0.3\textwidth}
    \centering
    \includegraphics{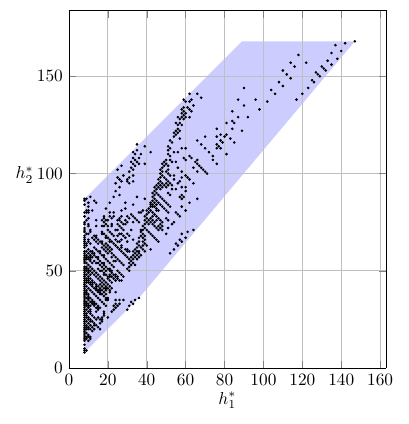}
    \caption{$\h_3 = 8$}
    \end{subfigure}
    ~
    \begin{subfigure}[t]{0.3\textwidth}
    \centering
    \includegraphics{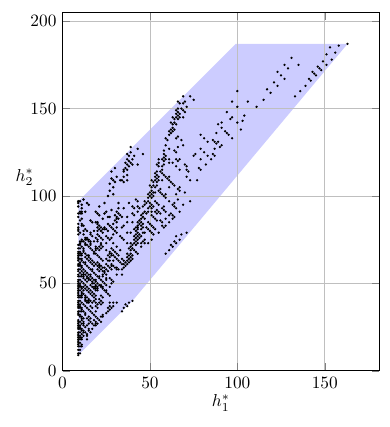}
    \caption{$\h_3 = 9$}
    \end{subfigure}
	\\
	\centering
    \begin{subfigure}[t]{0.3\textwidth}
    \centering
    \includegraphics{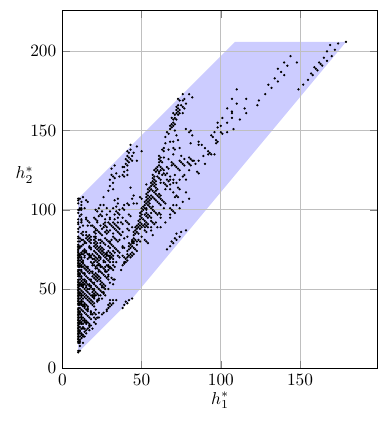}
    \caption{$\h_3 = 10$}
    \end{subfigure}%
    ~
    \begin{subfigure}[t]{0.3\textwidth}
    \centering
    \includegraphics{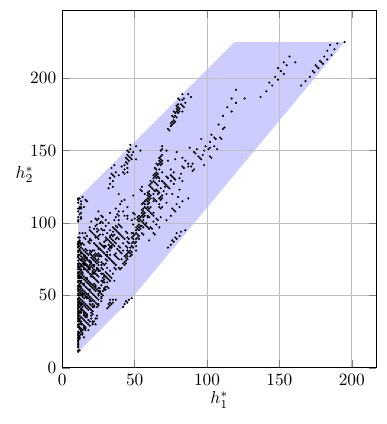}
    \caption{$\h_3 = 11$}
    \end{subfigure}
    \end{adjustbox}
\end{figure*}

\newpage
\section{Frequency of basic properties of lattice polytopes}
\label{app:tables}

\begin{center}
    \small
    \begin{longtable}{c@{\ \ }c@{\ \ }c@{\ \ }c@{\ \ }c@{\ \ }c@{\ \ }c}
    \caption{Number of three-dimensional polytopes (total, spanning, very ample,
	     IDP, having a unimodular cover, having a unimodular triangulation),
	     ordered by their normalized volume}
    \label{tab:W32}\\
    \toprule
    Volume&TOT&SP&VA&IDP&UC&UT\\
    \cmidrule(lr){1-1} \cmidrule(lr){2-7}
    \endfirsthead
    \multicolumn{7}{l}{\vspace{-0.7em}\tiny Continued from previous page.}\\
    \addlinespace[1.7ex]
    \midrule
	Volume&TOT&SP&VA&IDP&UC&UT\\
	\cmidrule(lr){1-1} \cmidrule(lr){2-7}
    \endhead
    \multicolumn{7}{r}{\raisebox{0.2em}{\tiny Continued on next page.}}\\
    \endfoot
    \bottomrule
    \endlastfoot
    \orow 1	 & 1       & 1		 & 1	   & 1		 & 1	& 1	   \\
    \erow 2	 & 3       & 2		 & 2	   & 2		 & 2	& 2	   \\
    \orow 3	 & 6       & 5		 & 5	   & 5		 & 5	& 5	   \\
    \erow 4	 & 17      & 15		 & 14	   & 14		 & 14	& 14   \\
    \orow 5	 & 19      & 17		 & 15	   & 15		 & 15	& 15   \\
    \erow 6	 & 54      & 51		 & 43	   & 43		 & 43	& 43   \\
    \orow 7	 & 59      & 57		 & 47	   & 47		 & 47	& 47   \\
    \erow 8	 & 154     & 147	 & 125	   & 125	 & 125	& 125  \\
    \orow 9	 & 181     & 177	 & 135	   & 135	 & 135	& 135  \\
    \erow 10 & 368     & 363	 & 291	   & 290	 & 290	& 290  \\
    \orow 11 & 414     & 411	 & 324	   & 323	 & 323	& 323  \\
    \erow 12 & 961     & 951	 & 748	   & 746	 & 746	& 745  \\
    \orow 13 & 1029    & 1025	 & 781	   & 779	 & 779	& 778  \\
    \erow 14 & 1929    & 1922	 & 1512	   & 1506	 & 1506	& 1506 \\
    \orow 15 & 2409    & 2403	 & 1843	   & 1837	 & 1837	& 1835 \\
    \erow 16 & 4254    & 4237	 & 3302	   & 3292	 & 3292	& 3288 \\
    \orow 17 & 4983    & 4978	 & 3801	   & 3787	 & 		& 	   \\ 
    \erow 18 & 8586    & 8574	 & 6656    & 6635	 & 		& 	   \\
    \orow 19 & 10186   & 10181	 & 7809	   & 7782	 & 		& 	   \\
    \erow 20 & 16708   & 16692	 & 13016   & 12971	 & 		& 	   \\
    \orow 21 & 20487   & 20479	 & 15630   & 15579	 & 		& 	   \\
    \erow 22 & 31163   & 31154	 & 24167   & 24085	 & 		& 	   \\
    \orow 23 & 37779   & 37773	 & 29271   & 29171	 & 		& 	   \\
    \erow 24 & 58906   & 58876	 & 45802   & 45663	 & 		& 	   \\
    \orow 25 & 70057   & 70049	 & 53907   & 53726	 & 		& 	   \\
    \erow 26 & 103117  & 103106	 & 80479   & 80225	 & 		& 	   \\
    \orow 27 & 126507  & 126495	 & 97652   & 97349	 & 		& 	   \\
    \erow 28 & 181732  & 181711	 & 141923  & 141488	 & 		& 	   \\
    \orow 29 & 219325  & 219317	 & 170327  & 169816	 & 		& 	   \\
    \erow 30 & 311917  & 311898	 & 243699  & 242984	 & 		& 	   \\
    \orow 31 & 376303  & 376295	 & 292843  & 291956	 & 		& 	   \\
    \erow 32 & 522559  & 522524	 & 409150  & 408010	 & 		& 	   \\
    \orow 33 & 636394  & 636382	 & 495472  & 494067	 & 		& 	   \\
    \erow 34 & 860937  & 860923	 & 675187  & 673321	 & 		& 	   \\
    \orow 35 & 1043226 & 1043214 & 816386  & 814161	 & 		& 	   \\
    \erow 36 & 1411304 & 1411272 & 1106938 & 1104038 & 		& 	   \\
    \end{longtable}
\end{center}

\begin{center}
    \small
    \begin{longtable}{c@{\ \ }c@{\ \ }c@{\ \ }c@{\ \ }c@{\ \ }c@{\ \ }c}
    \caption{Number of four-dimensional polytopes (total, spanning, very ample,
	     IDP, having a unimodular cover, having a unimodular triangulation),
	     ordered by their normalized volume}
    \label{tab:W32}\\
    \toprule
    Volume&TOT&SP&VA&IDP&UC&UT\\
    \cmidrule(lr){1-1} \cmidrule(lr){2-7}
    \endfirsthead
    \multicolumn{7}{l}{\vspace{-0.7em}\tiny Continued from previous page.}\\
    \addlinespace[1.7ex]
    \midrule
	Volume&TOT&SP&VA&IDP&UC&UT\\
	\cmidrule(lr){1-1} \cmidrule(lr){2-7}
    \endhead
    \multicolumn{7}{r}{\raisebox{0.2em}{\tiny Continued on next page.}}\\
    \endfoot
    \bottomrule
    \endlastfoot
    \orow 1  & 1      & 1      & 1      & 1      & 1    & 1    \\
    \erow 2  & 3      & 2      & 2      & 2      & 2    & 2    \\
    \orow 3  & 8      & 6      & 6      & 6      & 6    & 6    \\
    \erow 4  & 28     & 21     & 19     & 19     & 19   & 19   \\
    \orow 5  & 31     & 27     & 21     & 21     & 21   & 21   \\
    \erow 6  & 109    & 91     & 71     & 71     & 71   & 71   \\
    \orow 7  & 113    & 107    & 74     & 74     & 74   & 74   \\
    \erow 8  & 391    & 333    & 242    & 242    & 242  & 242  \\
    \orow 9  & 438    & 409    & 255    & 255    & 255  & 255  \\
    \erow 10 & 1019   & 956    & 618    & 618    & 618  & 618  \\
    \orow 11 & 1109   & 1094   & 664    & 664    & 664  & 664  \\
    \erow 12 & 3251   & 2993   & 1851   & 1850   & 1850 & 1849 \\
    \orow 13 & 3123   & 3103   & 1762   & 1761   & 1761 & 1760 \\
    \erow 14 & 6863   & 6680   & 3921   & 3918   &      &      \\
    \orow 15 & 8506   & 8327   & 4563   & 4560   &      &      \\
    \erow 16 & 17309  & 16681  & 9509   & 9500   &      &      \\
    \orow 17 & 18861  & 18826  & 10074  & 10066  &      &      \\
    \erow 18 & 38061  & 37224  & 20146  & 20125  &      &      \\
    \orow 19 & 42067  & 42023  & 22016  & 21997  &      &      \\
    \erow 20 & 80578  & 79132  & 42297  & 42253  &      &      \\
    \orow 21 & 94373  & 93832  & 47260  & 47214  &      &      \\
    \erow 22 & 158030 & 156975 & 81594  & 81501  &      &      \\
    \orow 23 & 184646 & 184580 & 92530  & 92429  &      &      \\
    \erow 24 & 330776 & 326283 & 165810 & 165631 &      &      \\
    \end{longtable}
\end{center}

\begin{center}
    \small
    \begin{longtable}{c@{\ \ }c@{\ \ }c@{\ \ }c@{\ \ }c@{\ \ }c@{\ \ }c}
    \caption{Number of five-dimensional polytopes (total, spanning, very ample,
	     IDP, having a unimodular cover, having a unimodular triangulation),
	     ordered by their normalized volume}
    \label{tab:W32}\\
    \toprule
    Volume&TOT&SP&VA&IDP&UC&UT\\
    \cmidrule(lr){1-1} \cmidrule(lr){2-7}
    \endfirsthead
    \multicolumn{7}{l}{\vspace{-0.7em}\tiny Continued from previous page.}\\
    \addlinespace[1.7ex]
    \midrule
	Volume&TOT&SP&VA&IDP&UC&UT\\
	\cmidrule(lr){1-1} \cmidrule(lr){2-7}
    \endhead
    \multicolumn{7}{r}{\raisebox{0.2em}{\tiny Continued on next page.}}\\
    \endfoot
    \bottomrule
    \endlastfoot
    \orow 1  & 1      & 1      & 1     & 1     & 1    & 1    \\
    \erow 2  & 4      & 2      & 2     & 2     & 2    & 2    \\ 
    \orow 3  & 10     & 6      & 6     & 6     & 6    & 6    \\ 
    \erow 4  & 38     & 23     & 21    & 21    & 21   & 21   \\ 
    \orow 5  & 42     & 33     & 25    & 25    & 25   & 25   \\ 
    \erow 6  & 169    & 115    & 86    & 86    & 86   & 86   \\ 
    \orow 7  & 163    & 144    & 90    & 90    & 90   & 90   \\ 
    \erow 8  & 659    & 475    & 322   & 322   & 322  & 322  \\ 
    \orow 9  & 707    & 600    & 344   & 344   & 344  & 344  \\ 
    \erow 10 & 1737   & 1465   & 841   & 841   & 841  & 841  \\ 
    \orow 11 & 1743   & 1685   & 869   & 869   & 869  & 869  \\ 
    \erow 12 & 6294   & 5022   & 2791  & 2791  & 2791 & 2790 \\ 
    \orow 13 & 5101   & 5007   & 2392  & 2392  &      &      \\ 
    \erow 14 & 12640  & 11533  & 5757  & 5756  &      &      \\ 
    \orow 15 & 15373  & 14315  & 6656  & 6655  &      &      \\ 
    \erow 16 & 34637  & 30638  & 14873 & 14870 &      &      \\
    \orow 17 & 32858  & 32650  & 14317 & 14314 &      &      \\
    \erow 18 & 77727  & 70953  & 32169 & 32160 &      &      \\
    \orow 19 & 75401  & 75103  & 32282 & 32272 &      &      \\
    \erow 20 & 167969 & 155336 & 68509 & 68488 &      &      \\
    \end{longtable}
\end{center}

\begin{center}
    \small
    \begin{longtable}{c@{\ \ }c@{\ \ }c@{\ \ }c@{\ \ }c@{\ \ }c@{\ \ }c}
    \caption{Number of six-dimensional polytopes (total, spanning, very ample,
	     IDP, having a unimodular cover, having a unimodular triangulation),
	     ordered by their normalized volume}
    \label{tab:W32}\\
    \toprule
    Volume&TOT&SP&VA&IDP&UC&UT\\
    \cmidrule(lr){1-1} \cmidrule(lr){2-7}
    \endfirsthead
    \multicolumn{7}{l}{\vspace{-0.7em}\tiny Continued from previous page.}\\
    \addlinespace[1.7ex]
    \midrule
	Volume&TOT&SP&VA&IDP&UC&UT\\
	\cmidrule(lr){1-1} \cmidrule(lr){2-7}
    \endhead
    \multicolumn{7}{r}{\raisebox{0.2em}{\tiny Continued on next page.}}\\
    \endfoot
    \bottomrule
    \endlastfoot
    \orow 1  & 1     & 1     & 1     & 1     & 1   & 1   \\
    \erow 2  & 4     & 2     & 2     & 2     & 2   & 2   \\ 
    \orow 3  & 11    & 6     & 6     & 6     & 6   & 6   \\  
    \erow 4  & 48    & 24    & 22    & 22    & 22  & 22  \\  
    \orow 5  & 51    & 36    & 27    & 27    & 27  & 27  \\  
    \erow 6  & 228   & 129   & 94    & 94    & 94  & 94  \\  
    \orow 7  & 204   & 167   & 97    & 97    & 97  & 97  \\  
    \erow 8  & 961   & 560   & 362   & 362   & 362 & 362 \\  
    \orow 9  & 970   & 728   & 392   & 392   & 392 & 392 \\  
    \erow 10 & 2444  & 1801  & 959   & 959   & 959 & 959 \\  
    \orow 11 & 2249  & 2092  & 964   & 964   & 964 & 964 \\  
    \erow 12 & 9872  & 6461  & 3362  & 3362  &     &     \\  
    \orow 13 & 6622  & 6334  & 2676  & 2676  &     &     \\  
    \erow 14 & 18069 & 14972 & 6684  & 6684  &     &     \\ 
    \orow 15 & 21837 & 18704 & 7828  & 7828  &     &     \\ 
    \erow 16 & 53513 & 41025 & 18006 & 18005 &     &     \\
    \end{longtable}
\end{center}

\end{document}